\documentclass[12pt]{article}

\usepackage{upgreek}
\usepackage{graphicx}
\usepackage{float}%

\usepackage{hyperref}
\hypersetup{colorlinks = true} 

\usepackage{amsmath}

\usepackage{amsfonts}
\usepackage{amssymb}
\usepackage{xypic}
\usepackage{enumerate}
\usepackage{subfigure}
\usepackage{amsthm}
\usepackage[russian, english]{babel}
\usepackage[OT2,T1]{fontenc}

\usepackage{mathrsfs}

\theoremstyle{plain}
\newtheorem{thm}{Theorem}[section]
\newtheorem{theorem}[thm]{Theorem}

\newtheorem{corollary}[thm]{Corollary}

\newtheorem{lemma}[thm]{Lemma}

\newtheorem{proposition}[thm]{Proposition}
\newtheorem{prop}[thm]{Proposition}
\newtheorem{definition}[thm]{Definition}
\newtheorem{assumption}[thm]{Assumption}

\theoremstyle{remark} 
\newtheorem{remark}[thm]{Remark}

\newcommand{\an}{\mathrm{an}}
\newcommand{\X}{\mathcal{X}}

\newcommand{\CZ}{\mathcal{Z}}
\newcommand{\CW}{\mathcal{W}}

\renewcommand{\phi}{\varphi}

\newcommand{\Div}{\mathrm{Div}}

\renewcommand{\O}{\mathscr{O}}
\renewcommand{\C}{\mathbb{C}}
\renewcommand{\d}{\mathrm{d}}

\renewcommand{\P}{\mathbf{P}}
\newcommand{\Q}{\mathbb{Q}}
\newcommand{\R}{\mathbb{R}}
\newcommand{\Z}{\mathbb{Z}}

\newcommand{\Hom}{\mathrm{Hom}}
\newcommand{\Pic}{\mathrm{Pic}}
\newcommand{\hPic}{\widehat{\mathrm{Pic}}}

\renewcommand{\div}{\mathrm{div}}
\DeclareMathOperator\Spec{Spec}

\newcommand{\hDiv}{\widehat{\mathrm{Div}}}

\newcommand{\eqv}{\mathrm{eqv}}
\newcommand{\Ban}{\mathrm{Ban}}
\newcommand{\MA}{\mathrm{MA}}

\renewcommand{\mod}{\mathrm{mod}}

\newcommand{\ord}{\mathrm{ord}}
\newcommand{\hyb}{\mathrm{hyb}}
\newcommand{\trop}{\mathrm{trop}}
\newcommand{\SF}{\mathrm{SF}}
\newcommand{\CH}{\widehat{CH}}
\newcommand{\et}{\text{\'et}}
\newcommand{\ppb}{\partial\bar{\partial}}

\newcommand{\M}{\mathcal{M}}

\title{Parametrization of geometric Beilinson--Bloch heights via adelic line bundles}
\author{Yinchong Song \\  \href{mailto:yinchongsong@stu.pku.edu.cn}{yinchongsong@stu.pku.edu.cn} }
\date{}

\begin{document}
	
\maketitle

\begin{abstract}
	Let $ S $ be a smooth quasi-projective variety over complex field $ \C $. For a smooth projective morphism $ \pi:X\to S $, we will introduce a new height pairing
	\begin{align*}
	CH^p_{\hom}(X/S) \times CH^q_{\hom}(X/S) \to \widetilde{\mathrm{Pic}}(S)
	\end{align*}
	with  $ \widetilde{\mathrm{Pic}}(S) $ the group of geometric adelic line bundles in the sense of Yuan--Zhang. It essentially parametrizes the asymptotic height pairing introduced by Brosnan and Pearlstein. We will show that this asymptotic height pairing coincides with Beilinson--Bloch pairing under certain conditions. 
\end{abstract}

\tableofcontents

\section{Introduction}

In \cite{Bei}, Beilinson introduced a (conditional) height pairing for homologically trivial cycles on a smooth projective variety $ X $ over a number field or a function field of one variable. In the function field case, it is a well-defined $ \Q_{\ell} $-valued pairing, and is conjectured to be $ \Q $-valued. In the number field case, we only know the existence of archimedean local height, and the non-archimedean part is still a conjecture. Bloch also introduced a height pairing in \cite{Blo}. In \cite{RS}, they generalize the $ \ell $-adic definition for high-dimensional base, and in \cite{Kah}, they give a refined height pairing using cycles. These two pairing only agree under certain conditions.

In our paper, we will introduce a new height pairing over the complex function field with target adelic line bundles. Let $ X,S $ be quasi-projective varieties over $ \C $, and $ \pi:X\to S $ be projective smooth of relative dimension $ n $. Let $ CH^p_{\hom}(X/S) $ be the subgroup of Chow group of $ X $ generated by subvarieties on $ X $ which are flat over $ S $ and have homologically trivial fibers. For $ p+q=n+1 $, we will define a pairing
\begin{align*}
CH^p_{\hom}(X/S) \times CH^q_{\hom}(X/S) \to \widetilde{\mathrm{Pic}}(S)
\end{align*}
which essentially parametrizes the asymptotic height pairing in \cite{BP}. It is functorial under pull-back by $ S'\to S $ where $ S' $ smooth quasi-projective. If $ K $ has transcendental degree $ 1 $, we will show that this height pairing agrees with Beilinson--Bloch height under certain conditions. 

The proof will be an analytic proof, both complex and Berkovich analytic. To begin with, we first study some point set topology property of the Berkovich analytic space, especially under the hybrid norm. Then we study Yuan--Zhang's adelic line bundles from a analytic point of view, and show that their theory behaves perfectly well under Berkovich analytification once we use hybrid norm to replace the usual norm in the local case. As a result, we get an analytic criterion of adelic line bundles. 

For any two cycles $ Z,W \in CH^p_{\hom}(X/S) $, Hain constructed a hermitian line bundle $ \bar{L}_{Z,W} $ on $ S^{\an} $ in \cite{Hai90}, which was later studied in \cite{BP}. Using our criterion, we see that $ \bar{L}_{Z,W} $ will determine a geometric adelic line bundle $ \tilde{L}_{Z,W}$ which parametrizes the asymptotic height pairing, thus we get our results. Finally, we show that asymptotic height pairing agrees with the geometric Beilinson--Bloch pairing under certain conditions.

\subsection{Berkovich spaces}

Let $ k $ be any noetherian ring equipped with a Banach semi-norm $ |\cdot|_{\Ban} $. Let $ U $ be a quasi-projective integral scheme, flat over $ k $. Then we obtain a Berkovich analytic space $ U^{\an} $ as in \cite{Ber09} containing all possible valuations of residue fields of points in $ U $ whose restriction to $ k $ is bounded by $ |\cdot|_{\Ban} $. It is a Hausdorff, locally compact topological space. 

In the global case, let $ (k,|\cdot|_{\Ban}) $ be either $ \Z $ with the usual norm or a field with the trivial norm. For a quasi-projective variety $ U $ over $ k $, we introduce two topological spaces related to $ U^{\an} $ in \S \ref{Berkovich spaces}. We define the \emph{interior part} $ U^{\beth}\subset U^{\an} $ to be the set of points where the reduction map 
$$ \mathrm{red}:U^{\an}\dashrightarrow U $$
is well-defined, and define the \emph{boundary part} 
$$ U^{b} :=  U^{\an}\backslash U^{\beth} $$  
to be the complement. Let $ \widetilde{U}^b:=U^{b}/\!\sim_{\eqv} $ be the quotient space of $ U^b $ modulo the norm-equivalence relation. The space $ U^{\beth} $ is closely related to the analytification of formal algebraic $ k $-schemes in \cite{Thu}. 

In the local case, let $ k $ is a complete field with a non-trivial valuation $ |\cdot| $, and let $ |\cdot|_0 $ be the trivial valuation on $ k $.  Define the hybrid norm $ |\cdot|_{\hyb}:=\max \{|\cdot|,|\cdot|_0\} $. We will show that $ (\Spec K)_{\hyb}^{\an}  $ is homeomorphic to the closed interval $ [0,1] $. For a quasi-projective variety $ U $ over $ k $, we have the hybrid analytification $ U^{\an}_{\hyb} $ with a structure map $ \pi:U^{\an}_{\hyb}\to (\Spec K)^{\an}_{\hyb}\cong [0,1] $. 

In practice, the space $ U^{\an}_{\hyb} $ could be viewed as a trivial fiber bundle  $ U^{\an}_t $ parametrized by $ t\in (0,1] $, and when $ t\to 0 $, the spaces $ U^{\an}_t $ ``converge'' to  $ U^{\an}_0 $, the Berkovich space over trivially valued field. We can define the spaces $ U^{\beth} $ and $ \widetilde{U}^b $ similarly as the global case. 

\begin{theorem}{(Theorem \ref{ThmBerkovichBoundary})}
	In both local and global cases above, the topological spaces $ U^{\beth} $ and $ \widetilde{U}^b $ are compact Hausdorff. 
\end{theorem}

In the local case, points in $ U^{\an}_t $ for $ t\in (0,1] $ do not have reduction in $ U $. Modulo equivalence, we see that $ \widetilde{U}^b $ would contain an open subspaces homeomorphic to $ U^{\an}_1 $, which is just the usual Berkovich analytic space over $ k $. We call $ \widetilde{U}^b $ the hybrid compactification of $ U^{\an} $. We will compare it with some other compactifications like \cite{AN,BJ} in \S \ref{HybridSpace}.

\subsection{Adelic line bundles over quasi-projective varieties}
Recently, Yuan and Zhang introduced adelic divisors and line bundles on quasi-projective varieties in \cite{YZ}. Let $ k=\Z $ or a field, and let $ X $ be a projective flat integral scheme over $ k $. If $ k=\Z $, let $ \hDiv(X) $ be the group of Cartier divisors on $ X $ with a conjugation-invariant Green function $ g_D $ on $ X_{\C}^{\an} $ which induce a hermitian metric $ \|\cdot\| $ on $ \O(D)_{\C}^{\an} $ on $ X_{\C}^{\an} $. If $ k $ is a field, this is called the geometric case, and by abuse of notation, $ \hDiv(X) $ is just $ \Div(X) $, the group of Cartier divisors.  

Now let $ U $ be quasi-projective flat integral scheme over $ k $. Let $ \hDiv(U/k) $ be the group of adelic divisors on $ U $ defined in \cite[\S 2]{YZ}. It is a reasonable limit of a sequence of divisors on projective models. Roughly speaking, to give an adelic divisor $ \bar{D}\in \hDiv(U/k) $, we need a usual divisor $ D\in \Div(U) $, and a sequence $ \bar{D}_i\in \hDiv(X_i)_\Q $ on different projective models $ X_i $ of $ U $, such that all restrictions $ D_i|_U $ are equal to $ D $, and the sequence $ \{\bar{D}_i\}_i $ satisfies a Cauchy condition on the boundary. See \S \ref{AdelicDivisors} for more precise definitions. 

Let $ \hDiv(U/k)_{b} $ be the subgroup of $ \hDiv(U/k) $ consisting of adelic divisors whose underlying divisor on $ U $ is trivial.

In \cite[\S 3]{YZ}, they also, define an analytification map of adelic divisors using Berkovich analytic space.

Now let $ k $ be either $ \Z $ or a field. If $ k=\Z $, let $ |\cdot|_{\Ban} $ be the usual absolute norm. If $ k $ is a field, let $ |\cdot|_{\Ban} $ be the trivial norm on $ k $. For $ U $ quasi-projective variety over $ k $, and $ U^{\an} $ be the analytification of $ U $.. 

A \emph{metrized divisor} $\bar{D}=(D,g_D) $ on $ U^{\an} $ consists of a Cartier divisor $ D\in \Div(U) $ and a continuous Green function 
$$ g_D:U^{\an}\backslash |D|^{\an}\to \R, $$ such that if $ D $ is locally defined by a function $ f\in K(U)^* $, then $ g_D(x)+\log |f(x)| $ locally extends continuously to $ |D|^{\an} $. The metric is called \emph{norm-equivariant} if for any $ x,y\in U^{\an}\backslash |D|^{\an} $ satisfying $ |\cdot|_x = |\cdot|_y^t $ for some $ t\in (0,\infty) $, then $ g_D(x)=tg_D(y) $. 

Let $ \hDiv(U^{\an})_{\eqv} $ be the group of norm-equivariant metrized divisors on $ U^{\an} $. Let $ \hDiv(U^{\an}){\eqv,b} $ be the subgroup of $ \hDiv(U^{\an})_{\eqv} $ consisting of metrized divisors $ (D,g_D) $ whose underlying divisor on $ U $ is trivial.  

Yuan and Zhang defined a Berkovich analytification map of adelic line bundles in \cite[\S 3]{YZ}, and obtain a canonical injection  $ \hDiv(U/k)\to \hDiv(U^{\an})_{\eqv} $ mapping $ \hDiv(U/k)_b $ into $ \hDiv(U^{\an})_{\eqv,b} $. These groups are equipped with the boundary topology.

In our paper, we will show that the above analytification maps are in fact bijective. 
\begin{theorem}\label{Theorem1}{(Theorem \ref{Theorem3.7})}
	Let $ k $ be either $ \Z $ or a field. Let $ U $ be a quasi-projective flat integral scheme over $ k $. Then the above canonical maps
	\begin{align*}
	\hDiv(U/k)\to \hDiv(U^{\an})_{\eqv},\\
	\hDiv(U/k)_b\to \hDiv(U^{\an})_{\eqv,b}
	\end{align*}
	are homeomorphisms. If the boundary norm on both sides are given by the same boundary divisor, they are isometries. 
	
	Further, under the canonical isomorphism $ \hDiv(U^{\an})_{\eqv,b} \cong C(U^{\an}_{\eqv}) $ which  maps the metrized divisor $ (0,g) $ to the function $ g $, the boundary topology is equivalent to the locally uniform convergence topology. 
\end{theorem}

Note that for non-normal variety $ U $, a metrized divisor $ (D,g_D) $ is not uniquely determined by its Green function. It may cause some problems if one tried to study metrized divisors purely analytically. 

It is not true for essentially quasi-projective variety, as explained in Remark \ref{Rem3.14}. 

Since adelic line bundles are limits of line bundles on projective models, it is equivalent to say that any continuous norm-equivariant metric on quasi-projective varieties can be approximated by model ones, which is a generalization of \cite[Theorem 7.12]{Gub}, and the proofs are similar.

We can also get similar results for local fields $ K $. This time it is different from Yuan--Zhang's one. 

Let $ K $ be either a complete non-archimedean field with a non-trivial valuation, or $ K=\R $ or $ \C $ with the standard absolute norm. The norm on $ K $ is denoted by $ |\cdot| $. Let $ |\cdot|_0 $ be the trivial norm on $ K $. Define the hybrid norm as $$ |\cdot|_{\hyb} = \max\{|\cdot|,|\cdot|_0\}. $$

Let $ U $ be a quasi-projective variety over $ K $. If $ K $ is non-archimedean, let $ \O_K $ be the valuation ring of $ K $; if $ K $ is archimedean, let $ \O_K=K $. Then Yuan--Zhang defined the group of local adelic divisors $ \hDiv(U/\O_K) $. Let $ U^{\an}=(U/(K,|\cdot|_{\hyb}))^{\an} $ be the Berkovich space of $ U $ over $ (K,|\cdot|_{\hyb}) $. The local version of Theorem \ref{Theorem1} is the following. 
\begin{theorem}\label{Theorem2}{(Theorem \ref{Theorem3.8})}
	The analytification map
	\begin{align*}
	\hDiv(U/\O_K)\to \hDiv(U_{\hyb}^{\an})_{\eqv}\\
	\hDiv(U/\O_K)_b\to \hDiv(U_{\hyb}^{\an})_{\eqv,b}
	\end{align*}
	are homeomorphisms. If the boundary norm on both sides are given by the same boundary divisor, they are isometries. 
	
	Further, under the canonical isomorphism $ \hDiv(U_{\hyb}^{\an})_{\eqv,b} \cong C(U_{\hyb}^{\an})_{\eqv} $ mapping the metrized divisor $ (0,g) $ to the function $ g $, the boundary topology is equivalent to the locally uniform convergence topology on $ U^{\an}_{\hyb} $. 
\end{theorem}
As a result, we get a analytic criterion of adelic line bundles. 

After that, we discuss the results on the toric varieties, and obtain a comparison theorem of toric adelic divisor in \S \ref{Toric}, special adelic divisors in \cite[\S 5]{Son}, and toroidal $ b $-divisors in \cite{BB}. In \S \ref{GlobalMA}, we discuss the intersection theory of adelic divisors and define the global Monge-Amp\`ere measure on $ \widetilde{U}^{b} $ for integrable adelic line bundles which generalizes the Monge-Amp\`ere measure in the complex case, and the Chambert-Loir measure in the non-archimedean case. The last application is a height pairing.

\subsection{Beilinson--Bloch height}

Let $ K $ be any field, and $ X $ be smooth projective over $ K $. Let $ CH^p(X) $ be the Chow group of cycles of codimension $ p $ in $ X $. Take a prime $ \ell\neq \mathrm{char}\ K $, and let $ CH^{i}_{\hom,\ell}(X) $ be the kernel of the $ \ell $-adic class map
\begin{align*}
\mathrm{cl_{\ell}}: CH^p(X)\to H^{2p}_{\et}(X_{\bar{K}},\Q_{\ell}(p)).
\end{align*}

If $ \mathrm{char}\ K=0 $, then we can use Betti cohomology instead, and it is independent of the prime $ \ell $. So we may drop $ \ell $ from the notation.

Now let $ K $ be the function field of a smooth projective curve $ B $ over an algebraically closed field $ k $. In \cite{Bei}, Beilinson introduced an $ \ell $-adic height pairing for Chow group of homologically trivial cycles of complementary codimension
\begin{align*}
CH^p_{\hom}(X)_{\Q} \times CH^q_{\hom}(X)_{\Q} \to \Q_{\ell}
\end{align*}
for $ p+q=n+1 $. He conjectures that this map is actually $ \Q $-valued, and gives a conditional construction in the number field case. 

In our paper, we will give a new pairing. Let $ S $ be quasi-projective variety over the complex field $ \C $, and $ \pi:X\to S $ be a smooth projective morphism of relative dimension $ n $. Let $ CH_{\hom}^p(X/S) $ to be the subgroup of Chow group $ CH^p(X) $ generated by codimension $ p $ cycles on $ X $, flat over $ S $, whose restriction to each fiber $ X_s $ is homologically trivial for each $ s\in S(\C) $. 
\begin{theorem}\label{HeightPairing}{(Theorem \ref{Theorem6.5})}
	There is a height pairing
	\begin{align*}
	CH^p_{\hom}(X/S) \times CH^q_{\hom}(X/S) \to \widetilde{\Pic}(S/\C)
	\end{align*}
	parametrizing the asymptotic height pairing in \cite{BP}. It is the unique pairing satisfying the following property. 
	\begin{enumerate}
		\item It is functorial for pull-back by morphisms $ S'\to S $ with $ S' $ a smooth quasi-projective variety over $ \C $. 
		
		\item If $ \dim S=1 $, it is compatible with Hain's biextension hermitian line bundle in \cite[3.4]{Hai90} and Lear's extension in his unpublished PhD thesis. 
	\end{enumerate}
\end{theorem}
See Theorem \ref{Theorem6.5} for a precise statement. 

Passing to limit on open subsets $ U\subset S $, we get the following version. 

\begin{corollary}\label{HeightPairingGeneric}
	Let $ K $ be a finitely generated field over $ \C $. Let $ X_{\eta} $ be a smooth projective variety over $ K $ of dimension $ n $. Then there is a pairing
	\begin{align*}
	CH^p_{\hom}(X_{\eta}) \times CH^q_{\hom}(X_{\eta}) \to \widetilde{\Pic}(\Spec K/\C)
	\end{align*}
	which is the direct limit of the pairing in \ref{HeightPairing}. 
\end{corollary}

To get a comparison with Beilinson--Bloch height, assume $ \dim S=1 $. Let $ \bar{S} $ be the smooth projective model of $ S $, and $ \bar{\pi}:\bar{X}\to \bar{S} $ be a projective flat morphism with $ \bar{X} $ smooth. Define $ CH^p(\bar{X})^0  $ to be the kernel of the composition
\begin{align*}
CH^p(\bar{X})\to H^{2p}(\bar{X},\Q) \to H^0(\bar{S},R^{2p}\bar{\pi}_*\Q). 
\end{align*}
It also occurs in Beilinson's construction \cite[1.2]{Bei} and is denoted by $ CH^p_{\text{\foreignlanguage{russian}{B}}}(X) $ in \cite[2C]{Kah}.

\begin{theorem}\label{Theorem1.6}{(Theorem \ref{Thm6.7})}
	Suppose that $ Z,W $ are flat cycles of $ X/S $, with homologically trivial fibers and disjoint support, of codimension $ p,q $ with $ p+q=n+1 $, such that one of them, say, $ Z $, admits an extension $ \tilde{Z} \in CH^*(\bar{X})^0 $. In this case their geometric Beilinson--Bloch pairing exists. Locally near a point $ s_0\in \bar{S}\backslash S $, choose a local coordinate, the archimedean Beilinson--Bloch height has the following asymptotic behavior
	\begin{align*}
	\langle Z_s,W_s \rangle_{s,\mathrm{ar}} = -\langle Z, W \rangle_{0,\mathrm{na}} \log |s| +O(1)
	\end{align*}
	for $ s $ close to $ 0 $. 
\end{theorem}
Here $ \langle Z_s,W_s \rangle_{s,\mathrm{ar}} $ means the archimedean Beilinson--Bloch height, and $ \langle Z, W\rangle_{0,\mathrm{na}} $ means the geometric Beilinson Bloch height over the function field $ K=K(S) $ and at the place $ 0 $. 

Combine the two theorems, we get the following. 

\begin{corollary}
	For each $ Z,W\in CH^*_{\hom}(X/S) $ of complementary codimension, the adelic line bundle $ \tilde{L}_{Z,W} $ in Theorem \ref{HeightPairing} has the following property. 
	
	For each curve $ C\subset S $, if $ Z|_{\pi^{-1}(C)} $ satisfies the same condition as in Theorem \ref{Theorem1.6}, then 
	\begin{align*}
	\deg_C  \tilde{L}_{Z,W}|_C = \langle Z|_{\pi^{-1}(C)}, W|_{\pi^{-1}(C)} \rangle_{K(C)}.
	\end{align*}
	Here $ \langle Z|_{\pi^{-1}(C)}, W|_{\pi^{-1}(C)} \rangle_{K(C)} $ is the geometric Beilinson--Bloch height over the function field of one variable $ K(C) $. 
\end{corollary}

We discuss the relationship between our pairings and the one in \cite{Kah}. The main difference is that ours are only defined for flat cycles, which work very well under pull-back maps. In \cite{Kah}, they only get a functorial under finite surjective pull-backs. One reason is that there pairing is essentially a codimension $ 1 $ theory, and it cannot give describe about higher codimension. 

To prove this result, we use Hain's result on Hodge theory to get a hermitian line bundle $ \bar{L} $ on $ S^{\an} $. Then we use our criterion of adelic line bundles to show that $ \bar{L} $ will determine a geometric adelic line bundle $ \tilde{L} $. 

It is natural to ask whether $ \bar{L} $ itself is adelic or not. For Gross-Schoen cycles, it is studied in \cite[3.3.4]{Yua} and  \cite{GZ}. They all conjectured  that the Beilinson--Bloch height in the number field case is parametrized by an adelic line bundle. Our pairing gives its function field analogue, and we believe that it could be generated to number field in the future.

\subsection*{Acknowledgement}
I would like to thank my advisor Xinyi Yuan about conjectures of the Beilinson--Bloch height, and his encouragement for this paper. I would also thank Zhelun Chen for our discussions during the seminars on Kyoto University and „Alexandru Ioan Cuza” University of Iași. Part of this work was highly inspired by his talk. The author would thank Marc Abboud, Ziyang Gao, Ruoyi Guo, and Shouwu Zhang for many helpful communications. 

The research was partially completed while the author was visiting the Institute for Mathematical Sciences, National University of Singapore and „Alexandru Ioan Cuza” University of Iași in 2025. We would like to thank these institutions for their hospitality.

\section{Berkovich spaces}\label{Berkovich spaces}

\subsection{Preliminaries on Berkovich spaces}\label{Preliminaries on Berkovich spaces}

Let $ k $ be a commutative ring with unity $ 1 $ which is complete with respect to a semi-norm $ |\cdot|_{\Ban} $. Such a pair $ (k,|\cdot|_{\Ban}) $ is called a \emph{Banach ring}. 
\subsubsection*{Affine case}
Let $ (k,|\cdot|_{\Ban}) $ be a Banach ring, and $ A $ be a $ k $-algebra. The \emph{Berkovich space} $ \M(A) $ is the topological space of multiplicative semi-norms $ |\cdot|:A\to \R $ whose restriction to $ k $ is bounded by $ |\cdot|_{\Ban} $. For a point $ x\in \M(A) $, the semi-norm is denoted by $ |\cdot|_x $, and for $ f\in A $, the semi-norm $ |f|_x $ is also denoted by $ |f(x)| $, and thus gives a real-valued function $ |f|:\M(A)\to \R $. The topology on $ \M(A) $ is defined as the weakest topology such that for each $ f\in A $, the function $ |f|:\M(A)\to \R $ is continuous. 

Two points $ |\cdot|_x,|\cdot|_y \in\ M(A) $ are called \emph{norm-equivalent} if there is a positive real number $ t\in \R_+ $ such that $ |\cdot|_x=|\cdot|_y^t $. 

We discuss two maps between $ \Spec A $ and $ \M(A) $. 

For each $ x\in \M(A) $, the kernel of the semi-norm $ |\cdot|_x:A\to \R $ is a prime ideal. This induces a continuous \emph{contraction map} 
$$ \rho:\M(A)\to \Spec A $$
sending a point $ x $ to $ \rho(x)=\ker |\cdot|_x $. Then the norm $ |\cdot|_x $ can be viewed as a norm on the residue field at $ \rho(x) $. Let $ H_x $ be the completion of the field $ (\kappa(\rho(x)),|\cdot|_x) $, called the \emph{residue field} at $ x\in \M(A) $. If $ H_x $ is non-archimedean, let $ R_x $ be the valuation ring of $ H_x $. 

The other one is the \emph{inclusion map} $ i:\Spec A\to \M(A) $ sending a point $ x\in \Spec A $ to the point $ (\kappa(x),|\cdot|_0) $, where $ |\cdot|_0 $ is the trivial norm on $ \kappa(x) $, i.e., $ |a|_0 = 1 $ for all non-zero $ a\in \kappa(x) $. This map is not continuous in general. 

If $ f:A\to B $ is a $k$-algebra homomorphism, $ x\in \M(B) $ corresponding to the multiplicative semi-norm $ |\cdot|_x $, then $ |\cdot|_x\circ f:A\to \R $ is a multiplicative semi-norm on $ A $, thus induces a continuous map $ f^*:\M(B)\to \M(A) $.  

In our paper, we only care about the topological structure of Berkovich spaces, so we do not discuss the notion of analytic functions. 

\subsubsection*{Scheme case}

In general, if $ (k,|\cdot|_{\Ban}) $ is a Banach ring, and $ X $ is a scheme over $ k $, then $ X $ can be covered by open affine subschemes $ \{\Spec A_j\} $, and we may glue $ \M(A_j) $ together to get the analytification $ X^{\an} $, which is more rigorously written as $ (X/(k,|\cdot|_{\Ban}))^{\an} $ if we need to emphasize the norm $ |\cdot|_{\Ban} $. If $ f:X\to Y $ is a morphism of schemes, then $ f $ induces a continuous morphism $ f^{\an}:X^{\an}\to Y^{\an} $, thus the analytification map is a functor from the category of schemes over $ k $ to the category of topological spaces. 

Let $ (k,|\cdot|_{\Ban}) $ be a noetherian Banach ring and $ X $ be a scheme of finite type. By \cite[Lemma 1.1, Lemma 1.2]{Ber09}, we have:
\begin{prop}
	(1) The space $ X^{\an} $ is locally compact. 
	
	(1) If $ X $ is separated over $ k $, then $ X^{\an} $ is Hausdorff. 
	
	(2) If $ X $ is proper over $ k $, then $ X^{\an} $ is compact.  
\end{prop}

We see that the contraction maps $ \rho_j:\M(A_j)\to \Spec A_j $ glue to a continuous map
$ \rho:X^{\an}\to X $, and the inclusion maps $ i_j: \Spec A_j \to \M(A_j) $ glue to  $ i:X\to X^{\an} $. The inclusion map $ i:X\to X^{\an} $ is not continuous in general.

\subsubsection*{Reduction map}
For a separated scheme $ X $ over $ k $, we have a reduction map, which is only partially defined on a subset of $ X^{\an} $. Assume that $ k $ is a Banach integral domain, and $ X $ is a noetherian separated scheme over $ k $. Let $ X^{\an} = (X/k)^{\an} $. There is a \emph{reduction map}
\begin{align*}
r:X^{\an}\dashrightarrow X
\end{align*}
which is only defined on a subset of $ X^{\an} $ as follows. 

For each non-archimedean point $ x\in X^{\an} $ such that the restriction of $ |\cdot|_x $ to $ k $ is bounded by the trivial semi-norm on $ k $, the canonical map $ \Spec H_x\to X\to \Spec k $ extends to $ \Spec R_x\to \Spec k $. Since $ X $ is separated, by the valuative criterion, there is at most one $ \phi:\Spec R_x\to X $ such that the diagram 
	\begin{align*}
\xymatrix{\Spec H_x \ar[r]\ar[d] & X \ar[d]\\
	\Spec R_x\ar[r]\ar@{.>}[ru]^{\phi} & \Spec k}
\end{align*}
commutes. If the morphism $ \phi $ does exist, then we say that the reduction map is defined at $ x $, and $ r(x) $ is the image of the closed points of $ \Spec R_x $ in $ X $ under the extension $ \phi $. 

\begin{definition}
	The interior part of $ X^{\an} $, denoted by $ X^{\beth}\subset X^{\an} $, is the set of all points in $ X^{\an} $ such that the reduction map $ r $ is defined at $ x $. The boundary part $ X^b $ is defined as $ X^b=X^{\an}\backslash X^{\beth} $, and the normalized boundary part is $ \widetilde{X}^{b}:= X^b/\!\sim_{\eqv} $. 
\end{definition}

\begin{remark}
	Here we explain the notations. If $ k $ is a field with trivial norm, in \cite[\S 1]{Thu}, Thuillier defined an analytic space $ \mathfrak{X}^{\beth} $ for any locally algebraic formal $ k $-scheme $ \mathfrak{X} $. Then $ {X}^{\beth} $ is the analytification of the formal scheme $ {X} $(which is just the scheme $ X $ with the discrete topology on the structure sheaf). 
	
	Let $ U $ be a quasi-projective integral flat scheme over $ k $, $ X $ be a projective model of $ U $  and let $ \mathfrak{X} $ be the formal completion of $ X $ along $ D=X\backslash U $. Then $ \mathfrak{X}^{\beth}=U^{b}\cup D^{\an} $, and $ U^{b}=\mathfrak{X}_\eta $ is equal to the generic fiber of the formal scheme $ \mathfrak{X} $ in \cite{Thu}. We choose this notation just because it only depends on $ U $, not on the projective model, and we want to emphasize that $ U^{\an} = U^{\beth}\cup U^{b} $. 
\end{remark}

If $ k $ is a noetherian Banach integral domain with the trivial norm, and $ X $ is proper, then the reduction map is defined on the whole space $ X^{\an} $, and the map $ r:X^{\an}\to X $ is anti-continuous \cite[Corollary 2.4.2]{Ber90}, namely, the inverse image of an open subset in $ X $ is closed in $ X^{\an} $. 

Now we state the main theorem we need about the Berkovich spaces. 
\begin{assumption}\label{Assumption}
	In this article, a Banach ring $ (k,|\cdot|_{\Ban}) $ will be one of the following. 
	
	\begin{enumerate}
		\item[(1)] Trivial norm case: $ k $ is any noetherian integral domain with the trivial norm $ |\cdot|_{\Ban} =|\cdot|_0 $.
		
		\item[(2)] Hybrid case: $ k $ is a field with a complete non-trivial valuation $ |\cdot| $, and take the hybrid norm $ |\cdot|_{\Ban} = |\cdot|_{\hyb}=\max\{|\cdot|,|\cdot|_0\} $. 
		
		\item[(3)] Global case: $ k=\Z $, and the norm $ |\cdot|_{\Ban} $ is the standard absolute norm on $ \Z $. 	
	\end{enumerate}
\end{assumption}
Note that in all cases, the trivial norm $ |\cdot|_0 $ is bounded by the Banach norm $ |\cdot|_{\Ban} $

\begin{theorem}\label{ThmBerkovichBoundary}
	Let $ k $ be as in Assumption \ref{Assumption}, and $ U $ be a quasi-projective integral flat scheme over $ k $. Then the topological spaces $ U^{\beth} $, $ \widetilde{U}^{b}$ are both compact Hausdorff. 
\end{theorem}

Before we give the proof of Theorem \ref{ThmBerkovichBoundary}, we give a short introduction of Berkovich spaces over fields with hybrid norms.

\subsection{Hybrid Berkovich spaces}\label{HybridSpace}

Now let $ K $ be a complete field with respect to a non-trivial norm $ |\cdot|_1 $, and let $ |\cdot|_0 $ be the trivial norm on $ K $. Define the \emph{hybrid norm} $ |\cdot|_{\hyb}:=\max\{|\cdot|_1,|\cdot|_0\} $. Then we can check that $ |\cdot|_t:= |\cdot|_1^t $ is a multiplicative norm for each $ t\in (0,1] $, and $ |\cdot|_1^0=|\cdot|_0 $. These norms are all bounded by $ |\cdot|_{\hyb} $, and it induces a continuous injection $ j:[0,1]\to \M(K,|\cdot|_{\hyb}) $. 

\begin{prop}
	The injection $ j: [0,1] \to \M(K,|\cdot|_{\hyb}) $ is a homeomorphism. 
\end{prop}
\begin{proof}
	Let $ |\cdot|_{\lambda} $ be any multiplicative norm on $ K $ which is bounded by $ |\cdot|_{\hyb} $. We will show that there is an element $ t\in [0,1] $ such that $ |\cdot|_{\lambda} = |\cdot|_1^t $. 
	
	If $ |\cdot|_{\lambda} $ is trivial, then we take $ t=0 $. Otherwise, there exist $ \pi\in K $ such that $ |\pi|_{\lambda}>1 $. Then $ |\pi|_{\hyb}\geqslant |\pi|_{\lambda}>1 $, hence $ |\pi|_1>1 $. There exist a unique number $ t\in (0,1] $, such that $ |\pi|_{\lambda} = |\pi|_{1}^t $.  
	
	For any non-zero $ x\in K $, if $ |x|_t\leqslant 1 $, then $ |x|_{\hyb} = 1 $ and hence $ |x|_{\lambda}\leqslant 1 $. Therefore, for any rational number $ a,b\in \Q $ with $ |\pi|_t^a\leqslant |x|_t \leqslant |\pi|_t^b $, we get $ |\pi|_{\lambda}^a\leqslant |x|_{\lambda} \leqslant |\pi|_{\lambda}^b $. But $ |\pi|_{\lambda} = |\pi|_t $, hence $ |x|_t=|x|_{\lambda} $ for all $ x\in K $. 
\end{proof}
To simplify the notations, we write $ K_t $ (resp. $ K_{\hyb} $) to be the Banach field $ K $ with the norm $ |\cdot|_t $ (resp. $ |\cdot|_{\hyb} $). In the followings we will identify $ \M(K_{\hyb}) $ with the closed interval $ [0,1] $ via $ j $. If some statement works for all $ K_t $ and $ K_{\hyb} $, we will write $ K_{\bullet} $ instead.

Now let $ U $ be any quasi-projective variety over $ K $. Then we have the Berkovich space $ (U/K_{\hyb})^{\an} $ with the structure map $ \pi:(U/K_{\hyb})^{\an}\to [0,1] $. The central fiber is $ \pi^{-1}(0)\cong (U/K_0)^{\an} $, and the fiber over $ 1 $ is $ \pi^{-1}(1)\cong (U/K_1)^{\an} $. For other fibers, we have the following results. 

\begin{prop}
	We have an homeomorphism $ \pi^{-1}((0,1]) \cong (U/K_1)^{\an}\times (0,1] $
\end{prop}
The proof of archimedean case is given in \cite{Ber09}. The author implicitly uses the following lemma. The proof of non-archimedean case is the same and even easier. 
\begin{lemma}
	Let $ t\in (0,1] $ be a real number, and $ A $ be an algebra over $ \R $. Let $ \|\cdot\| $ be a multiplicative semi-norm on $ A $ extending $ |\cdot|^t $ on $ \R $. Then $ \|\cdot\|^{1/t} $ is also a semi-norm on $ A $ extending $ |\cdot| $ on $ \R $. 
\end{lemma}
\begin{proof}
	We only need to check the triangle inequality. By direct computation, we get:
	\begin{align*}
	\|x+y\|^n&\leqslant\sum_{i=0}^n \left|{n\choose i}\right|^t \|x\|^i \|y\|^{n-i}\\
	&\leqslant (n+1)\left(\sum_{i=0}^{n} \left|{n\choose i}\right| \|x\|^{\frac{i}{t}} \|y\|^{\frac{n-i}{t}}\right)^t\\
	&=(n+1)\left( \|x\|^{\frac{1}{t}} + \|y\|^{\frac{1}{t}}\right)^{nt}
	\end{align*}
	Take $ nt $-th roots, and let $ n\to \infty $, we get
	\begin{align*}
	\|x+y\|^{\frac{1}{t}}\leqslant \|x\|^{\frac{1}{t}} + \|y\|^{\frac{1}{t}}
	\end{align*}
\end{proof} 

\subsubsection*{Compactifications}

The topological spaces $ (U/K_{\bullet})^{\an} $ with any norm on $ K $ as above are all Hausdorff and locally compact, but it is not compact unless $ U $ is projective over $ K $. For many purpose we need to obtain compact spaces, and hence we define various compactifications of these spaces, especially of $ (U/K_1)^{\an} $. 

The most common one is just given by a compactification of varieties. Namely, given an open immersion of varieties $ U\subset X $ with $ X $ projective, we obtain an compactifications $ (U/K_{\bullet})^{\an}\to (X/K_{\bullet})^{\an} $. Of course it depends on $ X $, so it is not intrinsic. One solution is that we can take all projective compactifications $ X $ and consider the inverse limit $ \displaystyle\lim_{X\supset U} (X/K_{\bullet})^{\an} $, which is a kind of Zariski-Riemann spaces. 

For the analytic space $ (U/K_1)^{\an} $, there is another compactification called hybrid compactification, as in \cite[\S 2]{BJ}. Although they only consider the case of complex manifolds, the essential idea works also for general valuation fields. 

In the case of complex manifold, they add some non-archimedean part as the boundary, usually the dual complex of boundary divisor. In \cite[Definition 4.4]{BJ}, they also consider the inverse limit of all such compactifications. This definition is essentially the same as ours defined later.

In \cite{AN}, they define the hybrid space of higher ranks, namely, the boundary itself contains both archimedean part and the non-archimedean part.

Using hybrid norms, we see that $ (U/K_{\hyb})^{\an} $ contains a canonical subset homeomorphic to $ (U/K_1)^{\an}\times (0,1] $, and thus by Theorem \ref{ThmBerkovichBoundary}, we have the following result. 

\begin{corollary}
	The topological space 
	$$ \widetilde{U}^b:=\left((U/K_{\hyb})^{\an}\backslash U^{\beth}\right)/\!\sim_{\eqv} $$
	is a compactification of $ (U/K_1)^{\an} $. Its boundary is the topological space $ \widetilde{U}_0^{b}$ derived from the Berkovich space $ (U/K_0)^{\an} $ over the trivially-normed field. 
\end{corollary} 

This compactification works for all complete field $ K $ with a non-trivial norm $ |\cdot|_1 $, and all quasi-projective varieties $ U $(not necessarily smooth).

We can also get other non-archimedean compactifications by retracting the boundary. Let $ U $ be a smooth quasi-projective variety with a smooth projective model $ X $, such that $ D=X\backslash U $ is a simple normal crossing divisor. Let $ \Delta(D) $ is the dual complex of $ D $. Thuillier has showed the following results in \cite[Proposition 4.7]{Thu}

\begin{theorem}\label{Theorem2.10}
	There is a canonical subset $ \mathcal{S}(U)^*\subset U_0^b $, homeomorphic to $ \Delta(D)\times (0,\infty) $, such that the topological space $ U_0^b $ retracts to $ \mathcal{S}(U)^* $. 
\end{theorem}

It can be shown in the proof that this retraction map is equivariant. Hence we get a continuous surjective map $ f:\widetilde{U}_0^b\to \Delta(D) $. 

If we take the pushout of the map $ f $ and the canonical injection $ \widetilde{U}_0^b\to \widetilde{U}^b $, then we get a compactification of the complex manifold $ (U/K_1)^{\an} $ with boundary $ \Delta(D) $. This reintroduces the work in \cite[\S 2]{BJ}.

\subsection{Proof of Theorem \ref{ThmBerkovichBoundary}}
First, we show that $ U^{\beth} $ is compact Hausdorff in all cases, and then so is $\widetilde{U}^b $ case by case.

\begin{prop}
	Let $ k $ be as in Assumption \ref{Assumption}. Then the interior part $ U^{\beth} $ is compact Hausdorff. 
\end{prop}
\begin{proof}
	Let $ X $ be a projective model of $ U $, and $ X^{\an} = (X/(k,|\cdot|_{\Ban}))^{\an} $. Since $ X $ is proper over $ k $, $ X^{\beth} $ is just the set of points in $ X^{\an}$ such that $ |a|_x \leqslant 1 $ for all $ a\in k $. Equivalently, we have
	\begin{align*}
	X^{\beth} \cong (X/(k,|\cdot|_0))^{\an},
	\end{align*}
	and hence $ X^{\beth} $ is compact Hausdorff. 	
	
	By \cite{Ber90}, the reduction map $ r:X^\beth \to X $ is anti-continuous, i.e., the inverse image of an open subset is closed. Thus $ U^\beth = r^{-1}(U) $ is a closed subspace of $ X^{\beth} $, hence it is compact Hausdorff. 
\end{proof}

Now we show that $ \widetilde{U}^b $ is compact Hausdorff case by case.

\subsubsection*{Trivial norm case}
In this subsection we prove Theorem \ref{ThmBerkovichBoundary} in the case when $ k $ is any noetherian integral domain with the trivial norm $ |\cdot|_{\Ban} = |\cdot|_0 $. 

Let $ U $ be a quasi-projective integral flat scheme over $ k $. Fix a projective model $ X/k $ of $ U $, and let $ D=X\backslash U $, with the reduced subscheme structure. Then the reduction map $ r:X^{\an}\to X $ is defined on the whole space $ X^{\an} $.

\begin{proof}	
	First we show the Hausdorffness of $ \widetilde{U}^{b} = U^{b}/\!\sim_{\eqv} $.
	
	Let $ x,y\in U^{b} $ be two points which are not norm-equivalent. Then $ r(x),r(y) $ are two points in $ D $(possibly equal). Since $ X $ is projective, there exists an open affine subscheme $ \Spec A\subset X $ containing $ r(x) $ and $ r(y) $. Then $ \Spec A $ also contains $ \rho(x) $ and  $ \rho(y) $, and $ x,y $ give two norm $ |\cdot|_x,|\cdot|_y $ on $ A $. 
	
	Since $ r(x),r(y)\in \Spec A $, we have $ |a|_x\leqslant 1 $, $ |a|_y\leqslant 1 $ for all $ a\in A $. 
	
	Note that $ \rho(y)\notin \overline{\{r(x)\}} $, as $ \rho(y)\notin D $ and $ r(x)\in D $. Thus we have 
	$$ \{a\in A:|a|_x < 1 \} \not\subset \{a\in A:|a|_y = 0\}.  $$
	By prime avoidance, two prime ideals $ \{a\in A:|a|_x= 0 \} $ and $ \{a\in A:|a|_y= 0\} $ cannot cover the ideal $ \{a\in A:|a|_x<1\} $. Thus there exists $ f_x\in A $ such that $ 0<|f_x|_x<1 $ and $ |f_x|_y\neq 0 $. 
	
	Similarly, take $ f_y\in A $ such that $ 0<|f_y|_y<1 $ and $ |f_y|_x\neq 0 $. Let $ f=f_x f_y $. Then we have $ 0<|f|_x<1 $ and $ 0<|f|_y<1 $. 
	
	Since $ x,y $ are not norm-equivalent, switch $ x,y $ if necessary, we may find $ g\in A $, such that
	\begin{align*}
	\frac{\log |g(x)|}{\log |f(x)|}< \frac{\log |g(y)|}{\log |f(y)|}.
	\end{align*}
	Take any $ a\in \R $ such that 
	$$ \frac{\log |g(x)|}{\log |f(x)|}<a< \frac{\log |g(y)|}{\log |f(y)|}, $$
	Define two open subset of $ (\Spec A)^{\an} $ as 
	\begin{align*}
	V_1=\{v\in (\Spec A)^{\an}:f(v)\neq 0,1;\ \frac{\log |g(v)|}{\log |f(v)|}<a \}
	\end{align*} 
	\begin{align*}
	V_2=\{v\in (\Spec A)^{\an}:f(v)\neq 0,1;\ \frac{\log |g(v)|}{\log |f(v)|}>a \}
	\end{align*}
	Here we assume that $ \log 0=-\infty<a $ for all real number $ a $. Then $ V_1 $, $ V_2 $ are two disjoint open subsets of $ (\Spec A)^{\an} $ which are stable under norm-equivalence relation and separate $ x,y $. Thus $ \widetilde{U}^{b} $ is Hausdorff. 
	
	For the compactness of $ \widetilde{U}^{b} $, since $ X^{\an} $ is compact Hausdorff, it is a $ T_4 $-topological space. As $ U^{\beth} $, $ D^{\an} $ are both closed in $ X^{\an} $, there exist disjoint open subsets $ W_1\supset U^{\beth} $ and $ W_2\supset D^{\an} $. Let $ W_3=X^{\an}\backslash(W_1\cup W_2) $, then $ W_3 $ is closed in $ X^{\an} $, and hence compact. 
	
	We show that $ W_3\to U^{b}\to \widetilde{U}^{b} $ is surjective. Given any $ x\in U^{b} $, choose an affine open subset $ \Spec A $ containing $ r(x) $, then $ x $ corresponding to a norm $ |\cdot|_x $ on $ A $.  Define the line $ L:[0,\infty]\to X^{\an} $ by  
	\begin{align*}
	L(t)=
	\begin{cases}
	i(\rho(x)),&\text{if } t=0;\\
	|\cdot|_x^t,&\text{if } t\in (0,\infty);\\
	i(r(x)),&\text{if }t=\infty,
	\end{cases}
	\end{align*}
	where $ i $ is the inclusion map $ X\to X^{\an} $. Then we can check that $ L $ is continuous. Since $ L(0)\in W_1 $ ,and $ L(\infty)\in  W_2 $, and $ [0,\infty] $ is connected, two disjoint open subsets $ L^{-1}(W_1) $ and $ L^{-1}(W_2) $ cannot cover $ [0,\infty] $. Thus there exist $ y\in W_3 $, and $ y $ is norm-equivalent to $ x $. Then $ W_3\to \widetilde{U}^{b} $ is surjective.
	
	Now $ \widetilde{U}^{b} $ is a continuous image of a compact space, and hence it is compact. 
\end{proof}

\subsubsection*{Hybrid norm case}
Now let $ k $ be the field with a non-trivial norm $ |\cdot| $, and consider the hybrid norm $ |\cdot|_{\hyb} $. Let $ U $ be a quasi-projective variety over $ k $, and $ U^{\an}=(U/K_{\hyb})^{\an} $ be the Berkovich space with structure morphism $ \pi:U^{\an}\to\M(k,|\cdot|_{\hyb})\cong  [0,1]  $, and let $ U_t^{\an}=\pi^{-1}(t) $ be the fiber over $ t $.

\begin{proof}Again we first show Hausdorffness. 
	
	Let $ x,y $ be two points in $ U^{b} $ which are not norm-equivalent. If $ \pi(x),\pi(y) $ are both non-zero, it follows that $ U_1^{\an} $ is open in $ U^{b}/\!\sim_{\eqv}$, and $ U_1^{\an} $ is Hausdorff. 
	
	If $ \pi(x)\neq 0 $ and $ \pi(y)=0 $, we fix an projective model $ X $, then the reduction map is defined at $ y $, and $ r(y)\in X\backslash U $. Choose an open affine subset $ \Spec A\subset X $ containing $ \rho(x) $ and $ r(y) $. The two prime ideals $ \{a\in A:|a|_x=0\} $ and $ \{a\in A:|a|_y=0\} $ cannot cover the ideal $ \{a\in A:|a|_y<1\} $. Thus there exists $ f\in A $ with $ |f|_x\neq 0 $ and $ 0<|f|_y<1 $. By multiplying a nonzero $ z\in k $, we may assume that $ |f|_x\neq 0,1 $ and $ |f|_y \neq 0,1 $ hold simultaneously.

	If $ \pi(x)=\pi(y)=0 $, similarly as the trivial norm case, we can also find $ f $ such that $ |f|_x\neq 0,1 $ and $ |f|_y \neq 0,1 $ hold simultaneously. The rest part of the proof are similar as the non-archimedean case. 
	
	To show the compactness, choose a projective model $ X $ and let $ D=X\backslash U $. Let $ Y=D^{\an}\cup X_1^{\an} $, which is closed subset in $ X^{\an} $. Then $ U^{\beth}, Y $ are both compact. Take disjoint open neighborhood $ W_1\supset U^{\beth} $, $ W_2\supset D^{\an} $, and let $ W_3=X^{\an}\backslash (W_1\cup W_2) $. Then similarly $ W_2 $ maps surjectively onto $ \widetilde{U}^{b} $, thus $ \widetilde{U}^{b}  $ is compact. 
\end{proof}

\subsubsection*{Global case}

For the case $ k=\Z $ with the standard absolute norm $ |\cdot| $, note that \begin{align*}
 \M(\Z,|\cdot|_0)\cong \M(\Z,|\cdot|_0) \cup \M(\Q,|\cdot|_{\hyb}),
\end{align*}
so it can be easily deduced by the two cases above.

\section{Adelic divisors on Berkovich spaces}\label{AdelicDivisors}
The adelic divisors and line bundles on quasi-projective varieties were introduced by Yuan--Zhang in \cite[\S 2]{YZ}. We just briefly review their definitions. 

\subsection{Adelic divisors}

Let $ k=\Z $ or a field, and let $ X $ be a projective flat variety over $ k $.

If $ k=\Z $, an \emph{arithmetic divisor} on $ X $ is a pair $ (D,g_D) $, where $ D $ is a Cartier divisor on $ X $, and $ g_D $ is a continuous real-valued conjugation-invariant function on $ (X\backslash D)(\C) $(with respect to the analytic topology), such that if $ D $ is locally defined by $ f $, then $ g_D(x)+\log |f(x)| $ extends continuously to $ D(\C) $. Equivalently, $ g_D $ induces a conjugation-invariant metric on the line bundle $ \O(D) $ over $ X(\C) $, such that the canonical section $ 1\in \O(D) $ has the norm $ \|1\|(x)=e^{-g_D(x)} $. The arithmetic divisor $ (D,g_D) $ is called \emph{effective} if $ D $ is effective and $ g_D $ is non-negative. 

If $ k $ is a field, by abuse of notation, an arithmetic divisor on $ X $ is just a usual divisor on $ X $. The set of arithmetic divisors on $ X $ is denoted by $ \hDiv(X) $. 

Now let $ U $ be a quasi-projective flat variety over $ k $, and let $ X $ be a projective model of $ U $. Define the \emph{group of $ (\Q,\Z) $-divisors} as
\begin{align*}
\hDiv(X,U):=\hDiv(X)_\Q \times_{\Div(U)_\Q} \Div(U).
\end{align*}
Thus a $ (\Q,\Z) $-divisor $ (\bar{D},D') $ consists of an arithmetic $ \Q $-divisor $ \bar{D}\in \hDiv(X)_\Q $ and a usual divisor $ D'\in \Div(U) $ such that their image in $ \hDiv(U)_\Q $ are the same. By abuse of notations, we just write $ \bar{D}\in \hDiv(X,U) $, and call $ D' $ the integral part of $ \bar{D} $. 

The \emph{group of model divisors} on $ U $ is defined as 
\begin{align*}
\hDiv(U)_{\mod}:= \lim_{\substack{\longrightarrow\\X}} \hDiv(X,U),
\end{align*}
where $ X $ ranges over all projective model of $ U $, and the morphisms are pull-back of divisors. 

There is a boundary norm on $ \hDiv(U)_{\mod} $, which more rigorously is an extended norm, i.e. it allows that $ |D|=\infty $ for some $ D $. Choose a fixed projective model $ X $ such that $ Z=X\backslash U $ is a divisor. If $ k=\Z $, we further choose a positive Green function $ g_Z $, and let $ \bar{Z}= (Z,g_Z) $. Let $ \bar{D}\in \hDiv(U)_{\mod} $. The \emph{boundary norm} $ \|\cdot\|_{\bar{Z}} $ is defined as 
\begin{align*}
\|\bar{D}\|_{\bar{Z}} = \inf\{ \varepsilon\in \Q_{>0}:-\varepsilon \bar{Z} \leqslant \bar{D}\leqslant \varepsilon \bar{Z} \},
\end{align*}
where the partial order is given by effectivity of divisors. If the integral part of $ \bar{D} $ is non-zero, such an $ \varepsilon $ does not exist, and we assume that $ \inf \emptyset =\infty $, thus $ \|\bar{D}\|_{\bar{\Z}} = \infty $ in this case.  

Different boundary divisors give different boundary norms on $ \hDiv(U)_{\mod} $, but they all induce the same boundary topology. The \emph{group of adelic divisor} $ \hDiv(U/k) $ is defined as the completion of $ \hDiv(U)_{\mod} $ with respect to any boundary norm. 

Let $ \bar{D}\in \hDiv(U/k) $. Then by definition, $ \bar{D} $ is represented by a sequence $ \bar{D}_i\in \hDiv(X_i) $, where $ X_i $ are projective models of $ U $. We may assume that if $ j>i $, then there is a morphism $ \pi_{ji}:X_j\to X_i $ compatible with open immersions $ U\subset X_i $, and $ (X_0,\bar{Z} = \bar{D}_0) $ is a boundary divisor defining the boundary norm. The sequence $ \bar{D}_i $ satisfies the Cauchy condition that there is a sequence $ \epsilon_i\in \Q_+ $ convergent to $ 0 $ such that
\begin{align*}
\epsilon_i \pi_{j0}^*\bar{Z} \leqslant \pi_{ji}^*\bar{D}_i-\bar{D}_j\leqslant \varepsilon_i \pi_{j0}^*\bar{Z}
\end{align*}
for all $ j>i $. In particular, $ D_i|_U $ are the same for all $ i $, and is denoted by $ D|_U $, called the integral part of $ \bar{D} $. 

Let $ \hDiv(U/k)_b\subset \hDiv(U/k) $ be the space of adelic divisors whose integral part is zero. These divisors are also called boundary divisors sometimes, but it is different with the one defining the boundary norm. Hopefully it won't cause confusions.

\subsection{Analytification of Adelic Divisors}
Let $ k $ be a commutative Banach ring which is also an integral domain, and $ X $ is an integral scheme over $ k $. An \emph{metrized divisor} (which is called arithmetic divisor in \cite[\S 3.3]{YZ}) $ \bar{D} $ on $ X^{\an} $ is a pair $ (D,g_D) $, where $ D $ is a usual divisor on $ X $, and $ g_D $ is a continuous function on $ X^{\an}\backslash D^{\an} $, such that if $ D $ is locally defined by a function $ f $, then $ g_D(x)+\log |f(x)| $ extends continuously to $ D^{\an} $. 

The metrized divisor $ \bar{D} $ is called \emph{norm-equivariant} if for any $ x_1,x_2\in X^{\an}\backslash D^{\an} $ satisfying $ |\cdot|_{x_1} = |\cdot|^t_{x_2} $ for some $ 0<t<\infty $, we have that $ g(x_1)=tg(x_2) $. 

Let $ \hDiv(X^{\an})_{\eqv} $ be the group of norm-equivariant arithmetic divisors. Yuan--Zhang defined an analytification map of adelic divisors. 

Let $ k $ be either $ \Z $ or a field. If $ k=\Z $, then $ |\cdot|_{\Ban} $ is the usual absolute norm on $ \Z $. If $ k $ is a field, then $ |\cdot|_{\Ban} $ is the trivial norm on $ k $. 

\begin{prop}\cite[Prop. 3.3.1]{YZ}
	Let $ X $ be a flat and quasi-projective integral scheme over $ k $, and let $ X^{\an}=(X/(k,|\cdot|_{\Ban}))^{\an} $. Then there is a canonical injection map
	\begin{align*}
	\hDiv(X/k)\to \hDiv(X^{\an})_{\eqv}.
	\end{align*}
\end{prop}

We briefly discuss their construction. First assume that $ X $ is projective. If $ k $ is a field, then there is a reduction map $ r: X^{\an}\to X $. Let $ D $ be a usual divisor on $ X $, and let $ x\in X^{\an} $. Take any open affine neighborhood $ U=\Spec A $ of $ r(x) $, such that $ D|_U $ is locally defined by $ f\in k(U)^* $. Define $ \tilde{g}(x)=-\log |f(x)| $. It is independent of the choice $ (U,f) $. Then $ (D,\tilde{g}) $ is the analytification of $ D\in \hDiv(X/k) $. 

If $ k=\Z $, let $ \bar{D}=(D,g)\in \hDiv(X) $. The analytification $ (D,\tilde{g}) $ is defined as follows. The restriction of $ \tilde{g} $ to the non-archimedean part of $ X^{\an} $ are defined in the same way as above. The archimedean part of $ X^{\an} $ is homeomorphic to $ X(\C)\times (0,1] $ modulo complex conjugaion, thus the restriction of $ \tilde{g} $ to the archimedean part is uniquely determined by $ g $. 

Let $ X $ be a fixed projective model of $ U $, such that $ D=X\backslash U $ is a divisor. Let $ \bar{D} $ be a boundary divisor defining the boundary norm, which induces a Green function $ g_{\bar{D}} $ on $ U^{\an} $. By the above construction, we have the following result.

\begin{proposition}
	For any $ x\in U^{\an} $, we have that $ g_{\bar{D}}(x)\geqslant 0 $, and $ g_{\bar{D}}(x)=0 $ if and only if $ x\in U^{\beth} $. 
\end{proposition}
In general, if $ U $ is quasi-projective, one may take limits. See \cite[\S 3]{YZ} for details.

Using Green functions on $ U^{\an} $, we can get another topological space which is homeomorphic to $ \widetilde{U}^b $. Roughly, instead of taking quotient, we choose a representative in each norm-equivalent class according the value of Green functions. Let $ U $ be quasi-projective variety over $ k $, $ X $ be a projective model of $ U $, and $ \bar{D} $ is a boundary divisor defining the boundary norm. 

\begin{definition}
	Let $ X,U, \bar{D} $ as above. The subspace $ \delta_{\bar{D}}(U) $ be the subspace of $ U^{\an} $ defined as follows. If $ k $ is a field, then 
	\begin{align*}
	\delta_{D}(U)=\{x\in U^{\an}:g_D(x)=1\}. 
	\end{align*}
	If $ k=\Z $, then it has an extra subset on the archimedean part
	\begin{align*}
	\delta_{\bar{D}} (U) =& \{ x\in U^{b}:g_{\bar{D}}(x)=1\}\\ &\cup \{x\in U^{b}: \max_{y\sim_{\eqv} x}{g_{\bar{D}}(y)} \leqslant 1, \text{ and } g_{\bar{D}} (x)  = \max_{y\sim_{\eqv} x}{g_{\bar{D}}(y)}\}. 
	\end{align*}
\end{definition}

\begin{proposition}
	The space $ \delta_{\bar{D}} (U) $ is compact Hausdorff.
\end{proposition} 

\begin{proof}
	If $ D $ is locally defined by a function $ f $, then $ g_{\bar{D}}(x) +\log |f(x)| $ extends continuously to $ D^{\an} $. Thus $ g_{\bar{D}}:U^{\an}\to \R $ extends continuously to $ \bar{g}_{\bar{D}}:X^{\an}\to \R\cup \{\infty\} $, and the space 
	\begin{align*}
	\{ x\in U^{b}:g_{\bar{D}}(x)=1\} = \bar{g}_{\bar{D}}^{-1}(1)
	\end{align*}
	is compact Hausdorff. The subspace $ \{x\in U^{b}:  g_{\bar{D}} (x)  = \displaystyle\max_{y\sim_{\eqv} x}{g_{\bar{D}}(y)}\leqslant 1 \} $ is homeomorphic to the space
	\begin{align*}
	\{x\in U/(\C,|\cdot|)^{\an}:g_{\bar{D}}(x)\leqslant 1\},
	\end{align*}
	which is also compact Hausdorff.
\end{proof}

\begin{corollary}
	The composition
	\begin{align*}
	\delta_{\bar{D}}(U)\to U^{b}\to \widetilde{U}^{b}
	\end{align*}
	is a homeomorphism.
\end{corollary}
\begin{proof}
	It is a continuous bijection from a compact space to a Hausdorff space, hence it is a homeomorphism. 
\end{proof}

\begin{corollary}
	If $ (k,|\cdot|_{\Ban}) $ is non-archimedean, then we have homeomorphism
	\begin{align*}
	U^b \cong \widetilde{U}^b\times \R
	\end{align*}
\end{corollary}

The space $ \delta_{\bar{D}}(U) $ is easier to study, but it depends on the choice of the boundary divisor $ \bar{D} $. The space $ \widetilde{U}^{b} $ is intrinsic. However, a Green function $ g_{\bar{D}} $ does not induce any reasonable function on $ \widetilde{U}^{b} $. We can only consider the function $ g_{\bar{D}_1}/g_{\bar{D}_2} $, so we still need to choose a boundary divisor.

Our main theorem is as follows.
\begin{theorem}\label{Theorem3.7}
	Let $ k $ be either $ \Z $ or a field. Let $ U $ be a flat quasi-projective integral scheme over $ k $. Then the above canonical maps
	\begin{align*}
	\hDiv(U/k)\to \hDiv(U^{\an})_{\eqv},\\
	\hDiv(U/k)_b\to \hDiv(U^{\an})_{\eqv,b} \cong C(U^{\an})_{\eqv}
	\end{align*}
	are homeomorphisms, where $ C(U^{\an})_{\eqv} $ is equipped with the locally uniform convergence, and all other groups are equipped with boundary topology. 
\end{theorem}

We give a quick proof of the homeomorphism $ \hDiv(U^{\an})_{\eqv,b} \to C(U^{\an})_{\eqv} $. 

\begin{prop}
	Fix a boundary divisor $ \bar{Z} $ defining the boundary norm with Green function $ g_Z $. Let $ \delta_{\bar{Z}}(U) $ as before. Then the restriction map $ g\mapsto g|_{\delta_{\bar{Z}}(U)} $ gives a homeomorphism 
	\begin{align*}
	C(U^{\an})_{\eqv} \to C(\delta_{\bar{Z}}(U)). 
	\end{align*}
	The sup norm $ C(\delta_{\bar{Z}}(U)) $ is isometric to the boundary norm on $ \hDiv(U^{\an})_{\eqv,b} $ with respect to $ \bar{Z} $, and is equivalent to the locally uniform convergence topology on $ C(U^{\an})_{\eqv} $
\end{prop}
\begin{proof}
	We construct the inverse map. Let $ g $ be any continuous function on $ \delta_{\bar{Z}}(U) $. We define the extension $ \tilde{g} $ on $ U^{\an} $ as follows. For $ x\in U^{\beth} $, define $ \tilde{g}(x)=0 $. For $ x\in U^b $, there is a unique point $ x'\in \delta_{\bar{Z}} (U) $ such that $ x'\sim_{\eqv } x $. We can extend $ \tilde{g} $ equivariantly to $ U^b $. 
	
	Now we check the continuity of $ \tilde{g} $. Since $ U^b $ is open, it suffices to check at $ x\in U^b $. Now $ \tilde{g}(x) = g_{\bar{Z}}(x)=0 $, and $ \tilde{g} $ is bounded by $ g_{\bar{Z}} $ up to constant, we get the continuity. 
	
	To check that our extension is unique, just note that if $ x\in U^{\beth} $, then for all $ t\in (0,\infty) $, $ |\cdot|_{x}^t\in U^{\beth} $, and $ U^{\beth} $ is compact. Thus any continuous equivariant function on $ U^{\an} $ vanishes on $ U^{\beth} $. 
	
	It's clear that the sup norm $ C(\delta_{\bar{Z}}(U)) $ is isometric to the boundary norm on $ \hDiv(U^{\an})_{\eqv,b} $ with respect to $ \bar{Z} $. Now we compare the sup norm on $ C(\delta_{\bar{Z}}(U)) $ with the locally uniform convergence topology on $ C(U^{\an})_{\eqv} $. Since $ \delta_{\bar{Z}(U)} $ is compact, for $ g_i\to g $ locally uniformly in $ C(U^{\an})_{\eqv} $, it converges uniformly on $ \delta_{\bar{Z}} $. On the other hand, if $ g_i\to g $ in $ C(\delta_{\bar{Z}}) $, let $ \tilde{g}_i,\tilde{g} $ be their equivariant extension on $ U^{\an} $, we get
	\begin{align*}
	|\tilde{g}_i (x) - \tilde{g}(x)|\leqslant \left( \sup_{y\in \delta_{\bar{Z}}} |g_i(y) - g(y)| \right) g_{\bar{Z}}(x). 
	\end{align*}
	The first term $ \sup_{y\in \delta_{\bar{Z}}} |g_i(y) - g(y)| $  converges to $ 0 $ as $ i\to \infty $, and the second term $ g_{\bar{Z}}(x) $ is locally bounded in $ U^{\an} $. Thus $ \tilde{g}_i $ is locally uniformly convergent to $ \tilde{g} $. 	
\end{proof}

\begin{proof}
	We fix a boundary divisor $ \bar{Z} $ defining the boundary norm with Green function $ g_Z $. Then the boundary norm on $ C(U^{\an}_{\eqv}) $ will be
	\begin{align*}
	\|g\|_{\bar{Z}} = \inf_{a>0}\{ -a g_Z\leqslant g\leqslant g_Z \}. 
	\end{align*}
\end{proof}

Let $ K $ be a field complete with respect to a non-trivial discrete valuation $ |\cdot| $. If $ K $ is non-archimedean, let $ O_K $ be the valuation ring of $ K $. If $ K $ is archimedean, let $ O_K=K $. Let $ U $ be a quasi-projective variety over $ K $. 

Yuan and Zhang also define the group of adelic divisors $ \hDiv(U/\O_k) $ in the this local case, and define injections $ \hDiv(U/\O_K)\to \hDiv((U/(K,|\cdot|))^{\an}) $. However, to get a similar results as in the global case, we need a different analytification using hybrid norms.

Let $ (K,|\cdot|) $ as above, and let $ |\cdot|_0 $ be the trivial norm on $ K $. Define the hybrid norm $ |\cdot|_{\hyb} = \max\{|\cdot|,|\cdot|_0 \} $ on $ K $, and let $ U^{\an}=(U/(K,|\cdot|_{\hyb}))^{\an} $. Similarly as in the global case, there is a natural analytification map $ \hDiv(U/\O_K) \to \hDiv(U^{\an})_{\eqv} $. 

\begin{theorem}\label{Theorem3.8}
	Notations are as above. Then the natural map $  \hDiv(U/\O_K) \to \hDiv(U^{\an})_{\eqv} $ is an isomorphism.  
\end{theorem}
We give a direct corollary here. Let $ U $ be a quasi-projective variety over a complete field $ K $ with a non-trivial norm $ |\cdot| $. Fix a projective model $ X $ with $ D $ a divisor, and let $ g_D $ be a positive continuous Green function of $ D $ on $ U_1^{\an} = (U/(K,|\cdot|))^{\an}  $. Let $ g:U_1^{\an}\to \R $ be any continuous function. 
\begin{corollary}
	The function $ g $ is a Green function of some adelic divisor if and only if the quotient $ g/g_{D}:U^{\an}\to \R $ extends continuously to $ \widetilde{U}^b $, where $ \widetilde{U}^b $ is viewed as a compactification of $ U_1^{\an} $. 
\end{corollary}

\subsection{Proof of Theorem \ref{Theorem3.7}}\label{Proof}
We first prove the case that $ k=\Z $  or a field with the trivial norm. Then we prove the hybrid case and the global case. It suffices to prove the case when the integral part is zero. 
\subsection*{Trivial norm case}
First we consider the trivial norm case, where $ k=\Z $ or a field with the trivial norm $ |\cdot|_0 $. 

Let $ X $ be a fixed projective model of $ U/k $, such that $ D=X\backslash U $ is a divisor. Then each ideal sheaf $ \mathcal{I}\subset \O_X $ defines a closed subscheme $ Y\subset X $. 

Now suppose that $ |Y|\subset |D| $. Such ideal sheaves are called  boundary ideal sheaves. Then $ \mathrm{Bl}_Y X $ ia also a projective model of $ U $, and the exceptional divisor $ E_Y $ is a model divisor in $ \Div(U)_{\mod} $. The analytification map gives an arithmetic divisor $ (0,g_Y) $ on $ U^{\an} $ with integral part zero, and a continuous function $ g_Y(x) $ on $ U^{\an} $. 

The function $ g_Y $ has the form
\begin{align*}
g_Y(x)=\min \{-\log |a(x)|: a\in \mathcal{I}_{r(x)}\}.
\end{align*}
It is identically zero on $ U^\beth $, and is always norm-equivariant and non-negative. If further that $ Y=D $, we see that $ g_D(x) $ is strictly positive on $ U^{b} $. 

Since $ g_Y(x) $ and $ g_D(x) $ are both norm-equivariant, the quotient $ h_Y(x)=\frac{g_Y(x)}{g_D(x)} $ is a norm-invariant continuous function on $ U^{b} $, thus it is the pull-back of a continuous function on $ \widetilde{U}^{b} $, which is still denoted by $ h_Y $. Such a function is called a \emph{model function} on $ \widetilde{U}^{b} $ associated to the closed subscheme $ Y $. If $ Y $ is defined by the boundary ideal sheaf $ \mathcal{I} $, we may also write $ g_{\mathcal{I}}=g_Y $ and $ h_{\mathcal{I}}=h_Y $. 

The constant function $ 1=h_D $ is a model function. Since we have
$$ h_{\mathcal{I}_1 \mathcal{I}_2} = h_{\mathcal{I}_1} + h_{\mathcal{I}_2} $$
and 
\begin{align*}
h_{\mathcal{I}_1 + \mathcal{I}_2} = \min\{h_{\mathcal{I}_1} , h_{\mathcal{I}_2}\},
\end{align*}
the space of model functions are closed under addition and taking minimum. To show that the $ \Q $-linear combination of model functions associated to the boundary ideal sheaf is dense in the Banach space of continuous functions on $ \widetilde{U}^{b} $, we need to show that model functions separate points.

\begin{prop}
	Model functions on $ \widetilde{U}^{b} $ separate points.
\end{prop}

\begin{proof}
	The proof here is almost the same as in \cite[Prop. 2.2]{BFJ}
	
	Let $ x,y\in U^{b} $ are two points which are not norm-equivalent. Choose a suitable $ x' $ which is norm-equivalent to $ x $ if necessary, we may assume that $ g_D(x)=g_D(y) $, and still $ x\neq y $. 
	
	If $ r(x)\neq r(y) $, we may assume that $ r(y)\notin \overline{r(x)} $. Let $ Y=\overline{r(x)} $. Then $ h_Y(x)>0 $ and $ h_Y(y)=0 $, thus $ h_Y $ separates $ x,y $. 
	
	Now assume that $ r(x)=r(y) $. Take an affine open neighborhood $ \Spec A $ of $ r(x) $.  Since $ x\neq y $, there exists $ f\in A $ such that $ |f(x)|\neq |f(y)| $. Consider the ideal sheaf $ \O_{\Spec A}\cdot f $, which can be extended to an ideal sheaf $ \mathcal{I} $ on $ X $. For each positive integer $ m $, the ideal sheaf $ \mathcal{I}_m=\mathcal{I}+(\mathcal{I}_D)^m $ is a boundary ideal sheaf on $ X $, and we have
	\begin{align*}
	h_{\mathcal{I}_m}=\min\{ \frac{-\log |f|}{g_D},m \}
	\end{align*}
	at $ x $ and $ y $, hence it separates $ x,y $ for $ m $ large. 
\end{proof}

By Boolean ring version of Stone-Weierstrass theorem, we get
\begin{corollary}\label{Corollary4.2}
	The $ \Q $-vactor space generated by model functions associated to boundary ideal sheaves is dense in $ C(\widetilde{U}^{b}) $. 
\end{corollary}

Thus it's ready to prove our main theorem for the non-archimedean case.

\begin{proof}
	It suffices to show that any continuous norm-equivariant function $ g $ on $ X^{\an} $ is a Green function of some boundary adelic divisors in $ \Div(X/k) $. 
	
	Fix a projective model $ X $ such that $ D_0=X\backslash U $ be  a divisor. Then $ h=g/g_{D_0} $ is a continuous norm-invariant function on $ U^{b} $. By the above proposition, $ h $ can be approximated by a sequence of model functions $ h_{E_i} $ uniformly on $ \widetilde{U}^b $, where $ E_i\in \Div(U)_{\mod} $ is a model boundary adelic divisor. 
	
	We claim that the sequence $ E_i $ is a Cauchy sequence. Since $ h_{E_i} $ is a Cauchy sequence in $ C(\widetilde{U}^{b}) $, we see that there exists $ \varepsilon_i\in \Q_{>0} $, such that for all $ j>i $
	\begin{align*}
	|h_{E_i}(x)-h_{E_j}(x)| \leqslant \varepsilon_{i}
	\end{align*}
	Thus we have
	\begin{align*}
	-\varepsilon_i g_{D_0}(x)\leqslant g_{E_i}(x) -g_{E_j}(x) \leqslant \varepsilon_i g_{D_0}(x)
	\end{align*}
	for all $ j>i $ and all $ x\in U^{b} $. Note that all these functions are zero in $ \widehat{U}^{\beth} $, hence it holds on $ U^{\an} $. 
	
	Use \cite[Lemma 3.3.3]{YZ}, possibly after a normalization, we get inequalities	
	\begin{align*}
	-\varepsilon_i D_0\leqslant E_i - E_j \leqslant \varepsilon_i D_0
	\end{align*}
	in terms of effectivities. Thus $ E_i $ converges to an adelic divisor $ E $. It is easy to check that the Green function of $ E $ is just $ g $. 
\end{proof}

\subsubsection*{Hybrid norm case and the global case}
Next we prove the hybrid norm case. Let $ U $ be a quasi-projective variety over $ K $, and $ X $ is a fixed projective model of $ U $ with boundary divisor $ D_0=X\backslash U $, and let $ \|\cdot\|_{\Ban} $ be the hybrid norm. 

Again, we may define the space of model functions on $ U^{b} $. The problem is that model functions themselves are not closed under taking minimum. We use Yuan--Zhang's results instead. 

First, we give a lemma, which works for both archimedean and non-archimedean local fields. Let $ K $ be a local field with valuation $ |\cdot| $, let $ |\cdot|_0 $ be the trivial norm on $ K $, and let $ |\cdot|_{\hyb}=\max\{ |\cdot|,|\cdot|_0 \}  $. Then $ \M(K,|\cdot|_{\hyb}) $ is homeomorphic to $ [0,1] $. 

Let $ U $ be a quasi-projective variety over $ K $, and consider  the structure morphism $ \pi:U^{\an}\to \M(K,|\cdot|_{\Ban})\cong [0,1] $. Let $ U_t^{\an} $ be the fiber over $ t $.  

\begin{lemma}
	Let $ g\in C(U^{\an}_1) $ be any continuous function. Extend $ g $ to $ U^{\an}_{(0,1]} $ norm-equivariantly. Then $ g $ extends continuously to a continuous function $ \tilde{g}\in C( U^{\an}) $ with $ \tilde{g}|_{U_0^{\an}} = 0 $ if and only if $ g $ grows as $ o(g_{D_0}) $ on $ U^{\an}_1 $ along the boundary. 
\end{lemma}

\begin{proof}
	Consider the norm-invariant function $ g/g_{D_0} $, viewed as a continuous function on $ U_1^{\an} $. Note that $ U^{\an}_1 $ is also an open subset of $ \widetilde{U}^{b}  $. 
	
	Now $ X^{\an}_1 $ and $ \widetilde{U}^{b} $ are two compactification of $ U_1^{\an} $, and the two conditions are both equivalent to $ g/g_{D_0} $ converges to zero when $ x $ approaches the boundary. 
\end{proof}
Fix a projective model $ X_0 $ of $ U $ with boundary divisor $ D_0=X_0\backslash U $. We also fixed a positive Green function $ g_{\bar{D}_0} $. Let $ X $ be any projective model of $ U $, and let $ \bar{D}=(D,g_D)\in \hDiv(X) $ with $ D|_U=0 $. Then $ \bar{D} $ induces a norm-equivariant continuous function $ \tilde{g}_{\bar{D}} $ on $ U^{\an} $. Let $ h_{\bar{D}}=g_{\bar{D}}/g_{\bar{D}_0} $ be the norm-invariant function on $ U^{b} $, which we viewed as a function on $ \widetilde{U}^{b} $. Such functions are called model functions. 

\begin{prop}
	The $ \Q $-linear combination of model functions are dense in $ C(\widetilde{U}^{b}) $. 
\end{prop}
\begin{proof}
	First, we show that model functions separate points. Let $ x,y\in \widetilde{U}^{b}$ be two different points. If one of them is archimedean,  we can only change the Green function on a small neighborhood in the archimedean part leaving all other place fixed. Otherwise, they are both non-archimedean, and we can use the results for non-archimedean case. 
	
	However, the model functions themselves are not stable under taking minimum. We need to check that if $ h_1,h_2 $ are two model functions, then $ \min\{ h_1,h_2 \} $ is contained in the closure of the space of model functions.
	
	Suppose that $ h_i $ is the model function of $ (D_i,g_i) $.
	Let $ \widetilde{U}^b_{0} $ be the non-archimedean part of $ \widetilde{U}^{b} $. Then $ \min\{h_1,h_2\}|_{\widetilde{U}^b_{0}} $ is a model function of some boundary divisor $ E $. Take any Green function $ g_E $ of $ E $. Now we consider the difference
	\begin{align*}
	\min\{h_1,h_2\} -h_E
	\end{align*}
	It is identically zero on the non-archimedean part. By \cite[Thm. 3.6.4]{YZ}, it is the Green function of some adelic divisor, which is the limit of a sequence of model functions.
	
	Now by Stone-Weierstrass theorem, we get the density result. 
\end{proof}
The rest part of the theorem is exactly the same as the non-archimedean case, and the case of $ \Z $ is deduced by the two cases above. 

\begin{remark}\label{Rem3.14}
	One may ask whether it is true for essentially quasi-projective variety $ U $. It's true for $ U=\Spec \Q $ as a scheme over $ \Z $, but not in general.
	
	Let $ f $ be a continuous and equivariant function on $ U^{\an} $.  It can be shown that $ f $ can be extended by zero to some $ V^{\beth} $ where $ V $ is a quasi-projective model of $ U $. For $ U=\Spec \Q $, we have $ V^{\beth}\cup U^{\an} = V^{\an} $, but this do not hold for general $ U $. 
	
	For example, consider the scheme $ \Spec \C(x) $. It has no archimedean points in analytification, so it can say nothing about it. Similarly for $ \Spec \Q_p(x) $ over $ \Q_p $, and $ f $ is identically zero on the trivial fiber $ (\Spec \Q_p(x))_{0}^{\an} $, we will have to check whether a continuous bounded function on $ (\Spec \Q_p(x))^{\an} $ could be extended to $ \mathbb{P}_{\Q_p}^{1,\an} $ almost everywhere, which is clearly not true in general. 
\end{remark}

\section{Adelic divisors on toric varieties}\label{Toric}
In this section we study a class of adelic divisors on toric varieties. Most background we need about toric varieties can be found in the books \cite{CLS,KKMS}. The field $ k $ in this section is assumed to be algebraically closed.

\subsection{Preliminaries on toric varieties}

Let $ T\cong (\mathbb{G}m)^n $ be a torus over a field $ k $. Let $ M=\Hom_{k-\text{alg. group}}(T,\mathbb{G}m) $ the character group, and $ N=\Hom_{k-\text{alg. group}}(\mathbb{G}m,T) $ the cocharacter group. For each $ m\in M $, let $ \chi^m $ be the corresponding invertible function on $ T $. Then $ \Gamma(T,\O_T) $ is generated by $ \{\chi^m\}_{m\in M} $ as vector space and we have a canonical isomorphism $ T\cong\Spec k[M] $.

There is a perfect pairing $ M\times N\to \Z $ given by $ (a,b)\mapsto \deg a\circ b:\mathbb{G}m\to \mathbb{G}m $. Then it induces the perfect pairing $ M_{\R} \times N_{\R}\to \R $. We will always identify $ M\cong N^{\vee} $ and $ M_{\R}\cong N_{\R}^{\vee} $. 

A toric variety $ X $ is a variety containing an open subset $ X_0\cong T $ such that the translation action $ T\times X_0\to X_0 $ extends to an action $ T\times X\to X $. There are deep connections between toric varieties and polytopes, polyhedra, and combinatorics. For example, we will frequently use the following facts. 
\begin{theorem}
There is a one-to-one correspondence between complete normal toric varieties and a finite rational polyhedral decomposition of $ N_{\R} $. 
\end{theorem}
See \cite[Theorem 6,7,8]{KKMS} for a proof. 

Let $ X $ be a complete normal toric variety. Each Cartier divisor supported on the boundary is also a Weil divisor, which is sum of codimension 1 subvarieties contained in $ X\backslash X_0 $. Since $ T $ acts on $ X $, it must be sum of codimension 1 orbits of $ T $. In particular, any boundary Weil divisor and Cartier divisors are $ T $-invariant. 

Now let $ D $ be a Cartier divisor $ D $ supported on the boundary. It will induce a supporting function on $ N_{\R} $ as follows. Recall that a point $ a\in N $ corresponds to a 1-parameter subgroup $ \lambda_a:\mathbb{G}m\to X $. Since $ X $ is complete, $ \lambda_a $ extends to a function $ \bar{\lambda}_a:\mathbb{A}^1\to X $, and $ \bar{\lambda}_a^*D $ is a Cartier divisor on $ \mathbb{A}^1 $ supported at the closed point $ 0 $. Its order is denoted by $  \ord_0\ \bar{\lambda}_a^* D $.

\begin{definition}
The supporting function $ \SF_D:N_{\R}\to \R $ is the unique continuous conical(meaning that $ f(\lambda x)=\lambda f(x) $ for all $ \lambda\in \R_+ $) function such that $ \mathrm{SF}_D(a)= \ord_0\ \bar{\lambda}_a^* D $ for all $ a\in N $. 
\end{definition}
To see that $ SF_D $ is well-defined, we refer to  \cite[\S 4.2]{CLS} and \cite[Theorem 9.]{KKMS}. Be careful that the supporting function in \cite[\S 4.2]{CLS} differs with ours by a sign.

\subsection{Analytification and tropicalization}
In this subsection we study some analytic property of toric varieties. Equipped $ k $ with the trivial norm, we get an analytic space $ T^{\an} $. Further, there is a \emph{tropicalization map} 
\begin{align*}
\mathrm{trop}: T^{\an}\to N_{\R}
\end{align*}
mapping a point $ |\cdot|_x $ to the linear functional $ (\chi^m\mapsto -\log |\chi^m|_x)_{m\in M} $. It is $ \R_{+} $-equivariant.

If we fix a coordinate $ T=\Spec k[T_1^{\pm 1},\ldots,T_n^{\pm 1}] $, then $ M $, $ N $ have canonical basis $ \{T_i\} $ and dual basis $ \{T_i^{\vee}\} $. Under this basis, the tropicalization map $ \mathrm{trop} $ has an explicit form
\begin{align*}
\mathrm{trop}: T^{\an}&\to \R^n\\
|\cdot|_x&\mapsto (-\log |T_1|_x,\ldots,-\log |T_n|_x). 
\end{align*}

The tropicalization map $ \trop:T^{\an}\to N_{\R} $ has a continuous section
\begin{align*}
\mathrm{emb}:N_{\R}\to T^{\an}
\end{align*}
sending the linear functional $ a:M\to \R $ to the norm 
$$ \left|\sum_{m\in M} a_m \chi_m\right| = \max_{m\in M} |a_m| \mathrm{e}^{-\langle m,a\rangle}. $$ 

Let $ q =\mathrm{emb}\circ \trop $. Then Thuillier showed in \cite[\S 2]{Thu} that $ q $ is a continuous retraction of $ T^{\an} $ to the its image $ q(T^{\an}) $. In particular, $ q^2=q $ and $ \trop(x) = \trop(q(x)) $.

We are interested in adelic divisors on $ T $ that are limit of T-invariant divisors on different toric varieties $ X $. First, we give a new description of toric divisors.

\begin{prop}\label{PropSupportingFunction}
	Let $ X $ be a complete toric variety, and $ D\in \Div(X) $ with $ D|_{T}=0 $. Let $ g_D:(X\backslash |D|)^{\an}\to \R $ be Green function of $ D $, which is necessarily equivariant. Then its restriction to $ T^{\an} $ is the pull-back of the supporting function $ \SF_D $ under the tropicalization map $ \trop:T^{\an}\to N_{\R} $. 
\end{prop}
\begin{proof}
	We first prove that for any $ x,y\in T^{\an} $ with $ \trop(x)=\trop(y) $, $ g_D(x)=g_D(y) $. Since $ q(x)=q(y) $, we may assume that $ y=q(x) $. 
	
	Let $ m:T\times X\to X $ be the action of $ T $ on $ X $, and $ p_2:T\times X\to X $ be the projection to the second factor. Note that $ D $ is $ T $-invariant. Therefore, for the two morphisms $ m,p_2:T\times X\to X $, we have $ m^*D=p_2^*D $. 
	
	Let $ g_D:(X\backslash |D|)^{\an}\to \R $ be the Green function of $ D $. Then the Green function of $ m^*D $ is just the pull back of $ g_D $ via the multiplication morphism
	\begin{align*}
	m:(T\times (X\backslash D))^{\an} \to (X\backslash D)^{\an}
	\end{align*}
	and similarly for $ p_2^*D $. Since $ m^*D=p_2^*D $(as usual divisors, not adelic divisors), for each point $ z $ in $ (T\times X)^{\beth}\backslash (T\times D)^{\beth} $, we must have $ g_D(m(z))=g_D(p_2(z)) $. Now it follows from Lemma \ref{LemmaToric} below, and $ g_D $ is the pull-back of some continuous conical function $ f $ on $ N_{\R} $. 
	
	To check that $ f=\SF_D $, for each $ a\in N $, the map $ \lambda_a: \mathbb{G}m\to T $ induces a map $ \lambda_a^{\trop}:N_{\mathbb{G}m,\R} \to N_{T,\R} $, which is just the linear map sending the canonical base in $ N_{\mathbb{G}m} $ to $ a\in N $, and the diagram	\begin{align*}
	\xymatrix{\mathbb{G}_m \ar[r]\ar[d]_{\lambda_a} & N_{\mathbb{G}m,\R} \ar[d]^{i_a}\\
		T\ar[r] & N_{T,\R}}
	\end{align*}
	commutes. Then it reduces to the case that $ T=\mathbb{G}m $, and then by direct computation. 
\end{proof}

\begin{lemma}\label{LemmaToric}
	Take any point $ x\in X_0^{\an}\cong T^{\an} $, then there is a point $ z\in (T\times X_0)^{\an} $, such that $ p_1(z)=q(x)\in T^{\beth} $, $ p_2(z)=x\in X_0^{\an} $ and $ m(z)=q(x) $. 
\end{lemma}
We will give an explicit proof, but the idea is the same as in \cite{Thu}. 
\begin{proof}
	Recall that $ q(x) $ is the norm given by the formula
	\begin{align*}
	\left|\sum_{m\in M} c_m \chi^m\right|_{q(x)} = \max_{m\in M} |c_m||\chi^m|_x 
	\end{align*}
	We check directly that the norm
	\begin{align*}
	\left|\sum_{m,n\in M} c_{mn}\chi^m \otimes \chi^n\right|_z = \max_{m\in M} \left(\left|\sum_{n\in M}c_{mn}\chi^n\right|_x  \right)
	\end{align*}
	satisfies our conditions. 
\end{proof}

\subsection{Toric adelic divisors}
Now we introduce the toric adelic divisors on $ T $ and then compare it to toroidal $ b $-divisors in \cite{BB} and special adelic divisors in \cite{Son}.

Let $ f:N_{\R} \to \R $ be any continuous conical function on $ N_{\R} $. Since $ \mathrm{trop} $ is equivariant, $ f\circ \trop $ is an equivariant function on $ T^{\an} $, and thus induces an adelic divisor. We will give a detailed description of such adelic divisors.

\begin{definition}
	A toric adelic divisor $ D $ on a torus $ T $ is an adelic divisor such that $ D|_T=0 $ and its Green function $ g_D $ is the pull-back of a continuous conic function  $ f $ on $ N_{\R} $ under the tropicalization map. 
\end{definition}

Similarly as the model case, if an adelic divisor $ D $ is toric, then the function $ f $ has an explicit description. 

\begin{prop}
	If $ D $ is a toric adelic divisor, then for each $ a\in N $ corresponding to the 1-dimensional parameter subgroup $ \lambda_a:\mathbb{G}m\to T $, we have
	\begin{align*}
	\lambda_a^* D = f(a)[0] + f(-a)[\infty],
	\end{align*}
	where $ [0] $ means the adelic divisor $ [0]\in \Div(\P^1)\subset \hDiv(\mathbb{G}_m) $, and similarly for $ [\infty] $. 
\end{prop}
The proof of it is th same as the model case. Hence the function $ f $ is also called the supporting function, and is denoted by $ \SF_D:N_{\R}\to \R $. 

Now we are ready to prove the main theorem. 

\begin{theorem}\label{Theorem 4.7}
	Let $ D $ be an adelic divisor with $ D|_T=0 $. The following two statements are equivalent. 
	
	(1) The adelic divisor $ D $ is toric: the Green function $ g_D $ is the pull back of th continuous conic function $ SF_D:N_{\R}\to \R $ under the tropicalization map $ \mathrm{trop}:T^{\an}\to N_{\R} $. 
	
	(2) There is a Cauchy sequence $ (X_i,D_i) $ representing $ D $ such that each $ X_i $ is a complete toric variety and $ D_i\in \Div(X_i) $ is supported in the boundary $ X_i\backslash T $. 
\end{theorem}

\begin{proof} 
	(2) $ \Rightarrow $ (1): By Proposition \ref{PropSupportingFunction} and the fact that the limit of toric adelic divisors are still toric. 
	
	(1) $ \Rightarrow $ (2): We fix a projective model $ X=\mathbb{P}^n $ of $ T $, with boundary divisor $ D_0=X\backslash T $, and consider the set of all  functions $ f:N_{\R}\to \R $ which are conical, continuous, piecewise linear, convex on each face, and has integral value on $ N $. By \cite[Theorem 9]{KKMS}, every such function gives a coherent sheaf of $ T $-invariant fractional ideals contained in $ i_*\O_T $, and then give a complete toric variety with a Cartier divisor on it whose supporting function is exactly $ f $. 
	
	Take any norm on $ N_{\R} $, and let $ \Sigma $ be the unit sphere in $ N_{\R} $. The restriction gives an isomorphism $ C(N_{\R})_{\eqv} $ and $ C(\Sigma) $. We will identify these two linear spaces. 
	
	Since the set of all such functions are closed under addition, taking minimum, and separating points in $ \Sigma $, by the Boolean ring version of Stone-Weierstrass theorem, any continuous conical function can be approximated by these functions uniformly on all compact subset. Hence any toric adelic divisor can be approximated by model toric divisors. 
\end{proof}
In fact, from the proof of Proposition \ref{PropSupportingFunction}, we see that $ D $ is toric if and only if $ D $ is $ T $-invariant, meaning that for the multiplication and projection maps $ m,p_2:T\times T\to T $, two adelic divisors $ m^*D $ and $ p_2^* D $ are equal with respect to the partial compactification $ T\times T\to T\times X $. Since we don't introduce adelic divisors for partial compactification, we refer to \cite[\S 3.1]{Son}.

\subsection{Comparison theorems}

\subsubsection*{Comparison with toroidal $ b $-divisors}
Now we compare this theorem with the toroidal $ b $-divisors in  \cite[\S 4.3]{BB}. These results work well for general toroidal embedding, but here we only consider the toric varieties to simplify the case.  Then any allowable modification of a toric variety is given by birational toric morphism \cite[Page 24,90]{KKMS}.

Let $ X $ be a complete toric variety with open subset $ U\cong T $. Let $ \Sigma $ be the set of all proper birational toric morphism $ X'\to X $ with $ X' $ a smooth complete toric variety. Then in each $ X' $, the Cartier divisor and the Weil divisors are the same. 

They define the toroidal Riemann--Zariski space as the inverse limit in the category of locally ringed spaces
$$ \mathfrak{X}_U:=\lim_{\substack{\longleftarrow\\X'\in \Sigma}} X' $$
where $ X'\to X $ are all proper toroidal birational morphism w.r.t the subdivision of $ N_{\R} $. Then they define the group of $ \R $-Cartier toroidal $ b $-divisors as 
\begin{align*}
\mathrm{CbDiv}(\mathfrak{X}_U)_{\R} := \lim_{\substack{\longrightarrow\\X'\in \Sigma}}\Div(X,U)_{\R}
\end{align*}
with maps given by pull-back of Cartier divisors. It is endowed with inductive topology. They also define the toroidal Weil divisor as 
\begin{align*}
\mathrm{WbDiv}(\mathfrak{X}_U)_{\R} := \lim_{\substack{\longleftarrow\\X'\in \Sigma}}\Div(X,U)_{\R}
\end{align*}
with maps given by push-forward of Weil divisors. It is endowed with the projective topology. 

Let $ \hDiv(U)_t $ be the group of toric adelic divisors. 
\begin{proposition}
	We have canonical injections
	\begin{align*}
	\mathrm{CbDiv}(\mathfrak{X}_U)\subset \hDiv(U)_{t} \subset \mathrm{WbDiv}(\mathfrak{X}_U)
	\end{align*}
	Such that the spaces of corresponding supporting function ae given by 
	\begin{align*}
	\mathrm{PL}(N_{\R}) \subset C(N_{\R})_{\mathrm{conic}}\subset \mathrm{Conic}(N_{\Q}),
	\end{align*}
	where $ \mathrm{PL}(N_{\R})  $ is the space of piecewise linear conical functions on $ N_{\R} $ whose linearity locus is defined over $ \Q $, $ C(N_{\R})_{\mathrm{conic}} $ is the space of continuous conical functions on $ N_{\R} $, and $ \mathrm{Conic}(N_{\Q}) $ is the space of arbitrary conical functions on $ N_{\Q} $. 
\end{proposition}

Be careful that their definition of $ \phi_D $ in \cite[Definition 4.2]{BB} differs with ours by a sign. Anyway, the vector space of supporting functions is the same.
\begin{proof}
	The inclusion map $ \mathrm{CbDiv}(\mathfrak{X}_U)\subset \hDiv(U)_{t}  $ is clear, and we only prove the second one $ \hDiv(U)_{t} \subset \mathrm{WbDiv} (\mathfrak{X}_U) $. 
	
	Suppose that the adelic divisor $ D $ is represented by $ (X_i,D_i) $, and all $ X_i $ are complete normal toric varieties. We may assume that $ D_0 = X_0\backslash U $ is the boundary divisor. We also assume that for $ j>i $, there exists a projection map $ \pi_{ji}:X_j\to X_i $. Then there is a sequence $ \epsilon_i \in \Q^+ $, converging to $ 0 $, such that 
	\begin{align*}
	-\epsilon_i\pi^*_{j0} D_0\leqslant \pi^*_{ji} D_i - D_j \leqslant \epsilon_i \pi_{j0}^* D_0. 
	\end{align*}
	
	Now let $ Y\in \Sigma $. For each $ i $, let $ D'_{Y,i}\in \mathrm{Div}(Y) $ be defined as follows. 
	
	Let $ Y'_i\in \Sigma $ such that we have the morphisms $ p_1:Y'_i\to Y $ and $ p_2:Y'_i\to X_i $. Then let $ D'_{Y,i}:=p_{1,*} p_2^* D_i $. We can check that $ D'_{Y,i} $ converges to some $ E_Y $ in the finite dimensional subspace of divisors in $ \Div(Y)_{\R} $ whose support is on the boundary, and such that $ \{E_Y\} $ form a Weil toroidal $ b $-divisor. The rest part of supporting functions is easy. 
\end{proof}

\subsubsection*{Comparison with special adelic divisors}
We can also compare toric adelic divisors with special adelic divisors in \cite[\S 5.]{Son}. We need to fix an smooth toric embedding $ T\subset X $ with $ X $ smooth complete. Then each Cartier divisor on $ X $ is also a Weil divisor. Let $ D=X\backslash T $, and let $ D=\sum D_j $ be the decomposition of $ D $ into irreducible components. Since each $ D_j $ is also a Weil divisor, there is a unique ray $ L_j $ in $ N_{\R} $ corresponding to $ D_i $. Let $ v_j\in N $ be a primitive generator of $ D_j $. 

We view $ \Spec k((t)) $ be the generic points in $ \Spec k[[t]] $, then we can talk about order of adelic divisors on $ \Spec k((t)) $. 

\begin{definition}
	An adelic divisor on $ T $ is called special (with respect to $ T\subset X $) if there is a continuous conic function $ f:N_{\R}\to \R $ such that for all morphism $ i:\Spec k((t))\to T $ with extension $ \Spec k[[t]]\to X $ (still denoted by $ i $), we have
	\begin{align*}
	\ord_0\  i^*E = f\left(\sum_{j} \ord_0 i^*D_j\cdot v_j \right). 
	\end{align*}
	Note that we always have $ \ord_0 i^*D_j\geqslant 0 $, and $ \{j:\ord_0 i^*D_j\neq 0\} = \{ j:i([0])\in D_j \} $ has cardinality at most $ n+1 $, hence the sum is in fact on a simplicial cone. 
\end{definition}
Such a function $ f $ is called the skeleton function in \cite{Son}. 

\begin{theorem}
	An adelic divisor is toric if and only it is special with respect to a smooth toric compactification. 
\end{theorem}

\begin{proof} 
	Any smooth toric compactifications are locally of the form $ \mathbb{G}m^n\subset \mathbb{A}^k\times \mathbb{G}m^{n-k} $, as it corresponding to a simplicial cone in $ N_{\R} $ generated by a subset of a basis $ N $. See also \cite[p. 29]{Ful}. Therefore, it suffices to consider the case that $ T\subset \mathbb{A}^n $ and that $ i:k[[t]]\to \mathbb{A}^n $. 
	
	Let $ E $ ba a toric adelic divisor, and $ i:\Spec k((t))\to T $ be  a map which extends to $ \Spec k[[t]]\to \mathbb{A}^n $. Take the valuation $ |\cdot|_s $ on $ k((t)) $ such that $ |t|=e^{-1} $. We also denote $ s $ to be the unique point in $ \M(k((t)),|\cdot|_s) $, and $ i^{\an}:\{s\}\to T^{\an} $ be its analytification. Then we have that $ \ord_0 i^*E = g_E(i(s)) $. 
	
	If $ E $ is toric, then the Green function $ g_E(i^{\an}(s)) $ of $ E $ at $ i(s) $ is determined by its tropicalization, ans we have
\begin{align*}
	\ord_0 i^*E &= g_E(i^{\an}(s)) = f(\trop(i^{\an}(s))) \\
	&= f\left(\sum_{j} \ord_0 i^*D_j \cdot v_j\right)
\end{align*}
Therefore, $ E $ is special. 

Conversely, if $ E $ is special, then we can find a toric adelic divisor with the same supporting function on $ N_{\R} $, so it suffices to prove the case that if $ E $ is special with supporting function zero, then $ E $ is itself zero. This follows from Lemma \ref{Lemma2} below. 
\end{proof}

\begin{lemma}\label{Lemma2}
	Let $ U $ be any quasi-projective variety over $ k $, and $ D $ be an adelic divisor on $ U $ with $ D|_U=0 $. Then $ D=0 $ if and only if for all smooth curves $ C $ and maps $ \lambda:C\to U $, the pull-back $ \lambda^*D=0 $. 
\end{lemma}
\begin{proof}
	Suppose that the adelic divisor $ D $ is represented by $ (X_i,D_i) $, and all $ X_i $ are normal. We may assume that $ D_0 = X_0\backslash U $ is the boundary divisor, and that for $ j>i $, there is a morphism $ \pi_{ji}:X_j\to X_i $. Then there is a sequence $ \epsilon_i \in \Q^+ $, converging to $ 0 $, such that 
	\begin{align*}
	-\epsilon_i\pi^*_{j0} D_0\leqslant \pi^*_{ji} D_i - D_j \leqslant \epsilon_i \pi_{j0}^* D_0. 
	\end{align*}
	
	Now assume that $ D\neq 0 $. Then for $ i $ large enough, the set
	\begin{align*}
	\{E\in \Div(X_i)_{\Q}: -\epsilon_i D_0\leqslant D_i - E \leqslant \epsilon_i  D_0\} 
	\end{align*}
	does not contain the zero divisor. In particular, there is a irreducible Weil divisor $ F $ such that either $  \ord_F D_i < -\epsilon_i\ord_F D_0  $ or $ \ord_F D_i > \epsilon_i\ord_F D_0 $.

	Now, choose a closed point $ x\in X_i $ such that $ x\in F $ but not in any other boundary irreducible divisors. Take a complete curve $ C' $ on $ X_i $ passing through $ s $ which is not contained in the boundary. Let $ C $ be the normalization of $ C'\cap U $ and $ \lambda:C\to X_i $ be the map. Then we can check that $ \ord_s \lambda^* D \neq 0 $. 
\end{proof}

\subsection{Local criterion of toric adelic divisors}
The theorems above can be easily generated to toroidal embedding. Let $ U $ be an open subvariety of a projective smooth variety $ X $ over $ K $ such that $ Z=X\backslash U $ is a simple normal crossing divisor. In this section, assume that $ Z $ is reduced, and each irreducible component is smooth.

Then on each closed point $ x\in Z $, there is an open neighborhood $ V $ of $ x $ and an \'etale morphism $ \phi:V\to \mathbb{A}^n $ such that $ \phi^{-1}(\mathbb{G} m^k\times \mathbb{A}^{n-k}) = V\cap U $. 

Let $ \Delta(Z) $ be the dual complex of $ |D| $. By \cite{Thu} or Theorem \ref{Theorem2.10}, there is a continuous contraction map $ \tilde{U}^b\to \Delta(Z) $. In particular, if we fix the boundary divisor $ Z $. any continuous function on $ \Delta(Z) $ induces a geometric adelic divisor $ \tilde{E} $ on $ U $. 

Let $ f $ be a continuous function on $ \Delta(Z) $ which induces a geometric adelic divisor $ \tilde{E} $. Suppose that there is a absolute norm $ |\cdot| $ on $ K $, we are trying to study the Green function of $ \tilde{E} $ on $ U_{\hyb}^{\an} $. Let $ U^{\an} $ be the usual analytification of $ U/(K,|\cdot|) $, viewed as a subspace of $ U_{\hyb}^{\an} $.

\begin{theorem}	\label{Theorem4.12}
	The function $ g $ on $ U^{\an} $ is a Green function of $ \tilde{E} $  if and only if the following holds. 
	
	For any $ x\in Z $, there is an analytic closed coordinate neighborhood $ V\cong \bar{\Delta}^n $ of radius $ r<1 $ in $ X^{\an} $ such that $ V\cap Z $  is defined by $ z_1\ldots z_r $. Then there is a simplex $ \Delta_x $ of dimension $ r-1 $ on $ \Delta(Z) $ associated with the point $ x $. Let $ f_{\Delta_x} $, there 
	
	\begin{align*}
	g(|\cdot|_z) = f_{\Delta_x} (-\log |T_1|_z,\ldots,-\log|T_r|_z) + o\left(-\sum_{i=1}^r \log |T_i|_z\right)
	\end{align*}
	for $ z\in V\cap U^{\an} $ approaches $ Z $, and $ \{T_i\} $ is the local coordinate of $ V $. 
	
	For $ k=\C $ with the usual norm, we can choose the usual analytic coordinate of disk $ \Delta^n $, and the formula becomes it is just 
	\begin{align*}
	g(z_1,\ldots,z_n) = f_{\Delta_x}(-\log |z_1|,\ldots,|z_r| ) + o\left(-\sum_{i=1}^r \log |z_i|\right). 
	\end{align*}
\end{theorem}

If $ K $ is non-archimedean, then as a set, polydisk is not the product of disk, so the formula seems different. 

\begin{proof}
	We need to check the continuity property on the hybrid compactification $ U_{\hyb}^{\an} $, or the normalized boundary part $ \tilde{U}_{\hyb}^{b} $. Since $ \tilde{U}_{\hyb}^b   = U^{\an}\cup \tilde{U}^{b}_{0} $, and the Green function on $ \tilde{U}^{b}_0 $ is the pull-back of $ f $ on $ \tilde{U}_0^{b}\to \Delta(Z) $, we only need to check the continuity on the space which is glued by $ \tilde{U}_{\hyb}^{b} $ and $ \Delta(Z) $ under the inclusion $ i:\tilde{U}_0^b \to U^{\an} $ and contraction map $ \pi: \tilde{U}_0^b \to \Delta(Z) $. Define
	\begin{align*}
	\tilde{U}_{\hyb,Z}^{b}:= \tilde{U}_{\hyb}^{b} \cup_{i,f} \Delta (Z).
	\end{align*}
	Now this new topological space is also a compactification of $ U^{\an} $ whose boundary is $ \Delta(Z) $, and it is just the compactification in \cite[\S 2]{BJ} locally. 
	
	Now let $ g $ be any continuous function on the usual compactification on $ U^{\an} $. Fix a positive Green function of $ Z $, their quotient $ g/g_Z $ is well-defined on $ U^{\an}\subset \tilde{U}_{\hyb}^{b} $. Thus $ g $ is an Green function of $ \tilde{E} $ if and only if $ g/g_Z $ extends continuously to $ \Delta(Z) $ with boundary function $ f $. 
	
	\'Etale locally, the variety $ U\subset X $ is isomorphic to $ \mathbb{G}m^r \times \mathbb{A}^{n-r} \subset \mathbb{A}^n $, thus the topological space $ \tilde{U}_{\hyb,Z}^{b} $ is locally homeomorphic to the space  $  \Delta^{*r}\times \Delta^{n-r}\cup S_r $, where $ S_r $ is a $ r $-dimensional simplex. 
	
	For $ \mathbb{G}m $, we can extend the tropicalization map $ \trop $ to the hybrid space
	\begin{align*}
	\trop_{\hyb}: \mathbb{G}m_{\hyb}^{\an} \to \R^n
	\end{align*}
	by the same formula $ x\mapsto (-\log |T_1|_x,\ldots,-\log |T_n|_x) $. In particular, the pull-back of $ f $ will give a continuous map on $ \mathbb{G}m_{\hyb}^{\an} $ by the formula	
	\begin{align*}
	f_{\hyb}(|\cdot|_z) = f(-\log |T_1|_z,\ldots,-\log |T_n|_z). 
	\end{align*}
	
	Now let $ g $ be any continuous function on $ \mathbb{G}m^{\an} $. Then $ g $ is a Grren function of $ f $ if and only if 
	\begin{align*}
	\frac{g}{g_Z}(|\cdot|_z) = \frac{f(-\log |T_1|_z,\ldots,-\log |T_n|_z)}{-\sum_{i=1}^r \log |T_1|_z} + o(1)
	\end{align*}
	as $ z $ approaches the boundary. \'Etale locally, on punctured poly disk $ g_Z $ has the form $ -\sum_{i=1}^r \log |T_1| $, where $ \{T_i\}_{i=1}^n $ is a regular coordinate at $ x $, and $ Z $ is defined by $ T_1\ldots T_r=0 $. Thus we get our result.

\end{proof}

\section{Global Monge-Amp\`ere Measure}\label{GlobalMA}
In this section we give an application of our results. It will not be used later.

Let $ k=\Z $ or a field, and let $ U $ be a quasi-projective flat variety over $ k $ of dimension $ n+1 $. Let $ \hPic(U)_{\mathrm{int}} $ be the group of integrable adelic line bundles defined by Yuan--Zhang. They also defined a intersection theory of adelic line bundles in \cite[\S 4]{YZ}, which is a symmetric multi-linear map on $ \hPic(U)_{\mathrm{int}}^{n+1} $. It can be extend as follows. 
\begin{prop}
	The intersection pairing extends uniquely to a multi-linear map
	\begin{align*}
	\hPic(U)\times \hPic(U)_{\mathrm{int}}^{n}\to \R,
	\end{align*}	
	such that if $ \bar{L}_1,\ldots,\bar{L}_{n}\in \hPic(U)_{\mathrm{int}} $ are fixed, then the map $ \hDiv(U)\to \R $ given by $ \bar{E}\mapsto (\bar{E}\cdot\bar{L}_1\cdots \bar{L}_{n}) $ is continuous with respect to the boundary norm on $ \hDiv(U) $. In other word, it is allowed to have at most one non-integrable adelic line bundle. 
\end{prop}
\begin{proof}
	Fix $ \bar{L}_1,\ldots,\bar{L}_{n}\in \hPic(U)_{\mathrm{snef}} $, and let $ E\in \hDiv(U)_{\mod} $ be an effective model divisor. Then $ \bar{E} $ is also integrable, and we have that
	\begin{align}\label{Formula1}
	(\bar{E}\cdot\bar{L}_1\cdots \bar{L}_{n})\geqslant 0.
	\end{align}
	
	Now let $ \bar{E}\in \hDiv(U) $ be an effective adelic divisor. Then $ \bar{E}=\lim_{i\to\infty} \bar{E}_i $ for effective model adelic divisors $ \bar{E}_i $. Fix a boundary divisor $ \bar{D}_0 $. Then there exist $ \varepsilon_i\in \Q_+ $ converging to $ 0 $, such that
	\begin{align*}
	-\varepsilon_i \bar{D}_0\leqslant \bar{E}_i-\bar{E}_j \leqslant \varepsilon_i \bar{D}_0
	\end{align*}
	in terms of effectivities for all $ j>i $. 
	
	Apply $ (-\cdot\bar{L}_1\cdots \bar{L}_{n}) $, by inequality (\ref{Formula1}), we have 
	\begin{align*}
	-\varepsilon_i (\bar{D}_0\cdot\bar{L}_1\cdots \bar{L}_{n})\leqslant (\bar{E}_i-\bar{E}_j\cdot\bar{L}_1\cdots \bar{L}_{n}) \leqslant \varepsilon_i (\bar{D}_0\cdot\bar{L}_1\cdots \bar{L}_{n})
	\end{align*}
	Thus the sequence $ (\bar{E}_i\cdot\bar{L}_1\cdots \bar{L}_{n}) $ is a Cauchy sequence, and we define 
	\begin{align*}
	(\bar{E}\cdot\bar{L}_1\cdots \bar{L}_{n}) = \lim_{i\to\infty} (\bar{E}_i\cdot\bar{L}_1\cdots \bar{L}_{n}).
	\end{align*}
	For general case, note that any adelic line bundles is the difference of two effective line bundles, we get our results. 
\end{proof}
In particular, let $ \hDiv(U)_b $ be the group of adelic divisors with zero integral part, then by Theorem \ref{Theorem3.7}, we have that $ \hDiv(U)_b\cong C(U^{\an})_{\eqv} $, which is further non-canonically isomorphic to $ C(\widetilde{U}^{b}) $ for the compact Hausdorff space $ \widetilde{U}^{b} $ defined in \S \ref{Berkovich spaces}. We thus get a morphism
\begin{align*}
\hPic(U)_{\mathrm{int}}^{n} \to (C(\widetilde{U}^{b}))^{\vee}\cong \mathcal{M}(\widetilde{U}^{b})
\end{align*}
where $ \mathcal{M}(U^{b}) $ is the Banach space of signed Radon measures on $ U^{b} $, and the latter isomorphism is given by the Riesz representation theorem for measures. We call it the \emph{global Monge-Amp\`ere measure} associated to the integrable adelic line bundles $ \bar{L}_1,\ldots,\bar{L}_{n} $, denoted by $ \mathrm{MA}(\bar{L}_1,\ldots,\bar{L}_{n}) $, and the intersection pairing is also a integration
\begin{align*}
(\bar{E}\cdot\bar{L}_1\cdots \bar{L}_{n}) = \int_{\widetilde{X}} h_{\bar{E}}(x)\  \mathrm{MA}(\bar{L}_1,\ldots,\bar{L}_n),
\end{align*}
where $ h_{\bar{E}} = g_{\bar{E}} / g_{\bar{D}_0} $ for the fixed boundary divisor $ D_0 $, and $ \mathrm{MA}(\bar{L}_1,\ldots,\bar{L}_n) $ is also computed via $ D_0 $. 

If all $ \bar{L}_i $ are equal, the measure is denoted by $ \mathrm{MA}(\bar{L}) $. 

To emphasize the independence of $ \bar{D}_0 $, we introduce the following notation.
\begin{align*}
\int_{U^{b}} g(x) \MA(\bar{L}_1,\ldots,\bar{L}_n) := \int_{\widetilde{U}^{b}} \frac{g(x)}{g_{\bar{D}_0}(x)}   \mathrm{MA}(\bar{L}_1,\ldots,\bar{L}_n),
\end{align*}
for any norm-equivariant measurable function on $ U^{b} $, where the measure $ \mathrm{MA}(\bar{L}_1,\ldots,\bar{L}_n) $ on $ \widetilde{U}^{b} $  is computed via the same boundary divisor $ \bar{D}_0 $.

\subsection*{Question}
Solve the equation $ \mathrm{MA}(\bar{L})=\mu $ for a given semi-positive measure $ \mu $ and a given line bundle $ L\in \Pic(U)_{\mathrm{snef}} $.

We have the following naive conditions for existence and uniqueness of the Monge-Amp\`ere equation. 

\begin{definition}
	Let $ \bar{E}\in \hDiv(U)_b $ be an integrable adelic divisor. Then $ E $ is called numerically trivial on the boundary if 
	\begin{align*}
	(\O(\bar{D})\cdot\O(\bar{E})\cdot\bar{L}_1\cdots\bar{L}_{n-1}) =0
	\end{align*}
	for all $ \bar{D}\in \hDiv(U)_b $ and all integrable adelic line bundles $ \bar{L}_i $. 
\end{definition}
We fix a model line bundle $ (X,\bar{L}_0) $ for $ (U,L) $. Then any other line bundle $ \bar{L} $ with $ \bar{L}|_U\cong L $ can be written as a sum of $ L_0 $ with a boundary divisor. 

Now consider the equation \begin{align}\label{FormulaMA}
\MA(\bar{L}) = \mu,
\end{align}
which is equivalent to solve the equation
$ \mathrm{MA}(\bar{L}_0 +f(x))=\mu $
for integrable norm-equivariant function $ f(x)\in C(U^{b})_{\eqv} $. 
\begin{prop}
	(1) If the equation (\ref{FormulaMA}) has a solution, then $ \bar{L}_0 $ and $ \mu $ satisfy the conditions
	\begin{align*}
	\O(\bar{D})\cdot \bar{L}_0 \cdots \bar{L}_0 = \int_{U^{b}} g_{\bar{D}} d \mu
	\end{align*}
	for all divisors $ \bar{D}\in \hDiv(U)_b $ which are numerically trivial on the boundary.
	
	(2) If $ \bar{L}_1 $ is a solution, then so is $ \bar{L}_1+\O(\bar{E}) $ for all adelic divisor $ \bar{E}\in \hDiv(U)_b $ which are numerically trivial on the boundary.  
\end{prop}

The proofs are easy. We call them the naive conditions for the Monge-Amp\`ere equation. It would be natural to ask whether these conditions are sufficient under some positivity conditions of $ L $ and $ \mu $.

\section{Beilinson--Bloch height}

\subsection{Preliminaries}

Let $ K $ be either a number field or a function field of smooth projective curve $ B $ over a field $ k $.  If $ K $ is a number field, let $ \O_K $ be its ring of integers, and let $ B=\Spec \O_K $.  

Let $ X_{\eta} $ be a smooth projective variety over $ K $ of dimension $ n $. Throughout this introductory section, we always assume the following. 
\begin{assumption}
	Suppose that $ X_{\eta} $ has a projective regular model $ X $ over $ B $. 
\end{assumption}
Later we will only consider the archimedean height and complex function field, so we can use Hironaka's resolution of singularities. Further, the concept homologically trivial will be independent of the cohomology theory. 

Let $ \widehat{CH}^*(X) $ be the arithmetic Chow group in the number field case, and the usual Chow group in the function field case. By \cite[\S 3]{SABK} and \cite{Ful98}, there is a pairing  
\begin{align*}
\langle\cdot ,\cdot\rangle_{X}:\widehat{CH}^p(X)_{\Q}\times \widehat{CH}^{n+1-p}(X)_{\Q}\to \R. 
\end{align*}

Let $ CH^p(X)^0_{\Q} $ denotes the subspace of $ \widehat{CH}(X)_{\Q} $ consisting of cycle classes whose restriction to $ X_{\eta} $ is homologically trivial. Let 
\begin{align*}
CH^p_{fin}(X)_{\Q} := \ker \left( \CH^p(X)_{\Q}\to \CH^p(X)_{\Q}\right)
\end{align*}
be the subgroup of Chow group generated by vertical cycles. 

Let $ \widehat{CH}^p_{fin}(X)^{\perp}_{\Q}\subset \widehat{CH}(X)^{n+1-p}  $ be the orthogonal complement to the subgroup $ \widehat{CH}^p_{fin}(X)_{\Q} $ with respect to the pairing above. Cycles in $  \widehat{CH}^p_{fin}(X)^{\perp}_{\Q} $ are called \emph{admissible} cycles. 

There is a restriction map
\begin{align*}
\phi_{\X} : \CH^p(X)^{\perp}_{\Q} \to \CH^{n+1-p}_{\hom}(X_{\eta})_{\Q}
\end{align*}
Then we can define the height pairing on the image of $ \phi_{X} $ by applying the pairing above to preimage under $ \phi_{X} $. 

A main conjectures about this pairing is whether $ \phi_X $ is surjective. If it is surjective(and assume the existence of regular model), then we will get a pairing
\begin{align*}
\CH^p_{\hom}(X_{\eta})_{\Q} \times \CH^{n+1-p}_{\hom}(X_{\eta})_{\Q}\to \R. 
\end{align*} 
This is the \emph{conditional Beilinson--Bloch height}. If one prefers, they could just use the image of $ \phi_{\X} $ and say that the Beilinson--Bloch height is well-defined on a smaller subgroup, and conjectures the smaller group is in fact the whole $ \CH^q_{\hom}(X)_{\Q} $. 

Now we assume that we have two cycles $ Z,W\in \CH^p_{\hom}(X) $, with disjoint support. Assume that one of them, say, $ Z $, has an admissible extension $ \mathcal{Z} $. Take any extension $ \CW $ of $ W $. Then the height pairing $ \langle\CZ,\CW \rangle_{X} $ is the sum of the local heights
\begin{align*}
\langle Z, W \rangle_{X} = \sum_{v\in M_K} \langle \CZ,\CW \rangle_{X, v},
\end{align*}
where $ M_K $ is the set of all places in $ K $. Note that the existence of local height is a local problem. 

It is known that the local height exists when $ v $ is a archimedean place, or $ X $ has good reduction at $ v $, see \cite{Bei}. For the archimedean place, we give a more detailed description. 

\subsection{Archimedean height pairing}

Let $ X $ be a smooth projective variety over $ \C $, and $ Z, W $ be two homologically trivial cycles on $ X $ of complementary codimension $ p,q $ with $ p+q=n+1 $. Assume that they have disjoint support. There exists a Green current of type $ (p-1,p-1) $ of $ Z $, smooth away from the support of $ Z $, satisfying the current equation
\begin{align*}
\ppb g_Z = \pi i \delta_Z. 
\end{align*}
Then the archimedean height is defined by the formula
\begin{align*}
\langle Z,W \rangle_{\infty} : = -\int_{W} g_Z. 
\end{align*}

The archimedean pairing is studied by Hain in \cite{Hai90} using Hodge theory. Given $ Z,W $ as before, he constructed a mixed Hodge structure with graded pieces $ \Z(1) $, $ H =  H^{2p-1}(X,\Z(p)) $, $ \Z $, called the biextension associated to $ Z $ and $ W $. He should that biextension is parametrized by non vanishing point on the Poincare line bundle $ B(H) $ on torus $ JH\times J\hat{H} $, where $ JH $ is the intermediate Jacobian, and $ \hat{H} $ is the dual of $ H $. Further, he constructed a canonical metric on $ B(H) $ using Hodge theory, and showed that the archimedean height can be calculated via Hodge theory 
\begin{align*}
\langle Z,W \rangle_{\infty} = \log \|b_{Z,W}\|
\end{align*}
where $ b_{Z,W} $ is the point in $ B(H) $ associated to the biextension given by $ Z,W $.

Now we come to the family version. Let $ S $ is quasi-projective over $ \C $, and $ \pi:X\to S $ be a smooth projective morphism. Take $ Z,W $ be two cycles in $ X $ of complementary codimension $ p,q $ with $ p+q=n+1 $, flat over $ S $, with generically disjoint support. Hain's construction works well for family, and they get a hermitian line bundle $ \bar{L}_{Z,W} $ on the complex manifold $ S^{\an} $ as follows. 

For each $ p $, the family $ \pi:X\to S $ will induce a variation of Hodge structure $ V= R^{2p-1}f_*\Q(p) $ on $ S^{\an} $, thus a toric bundle $ J^p $ over $ S^{\an} $. It parametrized extension of variation of Hodge structure of $ R^{2p-1}f_*\Q(p) $ by $ \Z $. The Poincare duality of cohomology groups will induces a a dual of torus bundle $ J^{2p-1}\cong (J^{2q-1})^{\vee} $. Usually for $ p\neq 1,n $ the torus $ J $ is not algebraic, so this is a purely analytic construction. 

Hain construct a Poincare line bundle on $ P $ on the torus bundle $ J H\times J\check{H} $ whose non-vanishing sections parametrize biextension of Hodge structure. This line bundle is equipped with a canonical metric $ \|\cdot\| $, which reflects how far is the biextension mixed Hodge structure from a split one. 

Now given a flat cycle $ Z $ on $ X $ with homologically trivial on each fiber, we will get a section $ \nu_Z:S\to J $, usually called a normal function. It is a holomorphic, horizontal, with admissible conditions on the behavior at infinity, see \cite[\S 7]{Voi03}. 

If $ Z,W $ are two flat cycles, then we get $ \nu_Z\times \nu_W:S\to J^p\times J^{q} $. Define the hermitian line bundle $ \bar{L}_{Z,W}:= (\nu_Z,\nu_W)^* \overline{P} $. The following result can be found in \cite[\S 3.4]{Hai90}.
\begin{theorem}[Hain]
	Give two flat cycles $ Z,W $ with generically disjoint support, there is a line bundle $ \bar{L}_{Z,W} $ with a meromorphic section $ \mu_{Z,W} $, such that 
	\begin{align*}
	\langle Z_s,W_s \rangle = -\log \|\mu_{Z,W}\|_s
	\end{align*}
	whenever $ Z_s $ and $ W_s $ have disjoint support. 
\end{theorem}

\begin{theorem}\label{Theorem Hain}
	The line bundle $ \bar{L}_{Z,W} $ depends only on the cycle class of $ Z,W $ in Chow group, and thus we get a pairing
	\begin{align}\label{ArchPairing}
	CH^p_{\hom}(X/S) \times CH^q_{\hom}(X/S) \to \hPic(S^{\an}).
	\end{align}
\end{theorem}
\begin{proof}
	By \cite[\S 9.2.3]{Voi03}, the normal function $ \nu_Z $ depends only on the cycle class on the fiber $ [Z_s] $. Since $ Z $ is flat over $ S $, for each $ s\in S $, the restriction of cycle $ f^*:CH^p_{\hom}(X/S)\to \CH^p_{\hom}(S_s) $ and the restriction of scheme coincide. 
\end{proof}

Later, in \cite{BP}, they should that the underlying line bundle $ L_{Z,W} $ is algebraic and the meromorphic section $ \mu_{Z,W} $ is a rational. In other words, it extends to some smooth compactification $ \bar{S} $ by Serre's GAGA principle. 

However, the metric on $ L_{Z,W} $ cannot  extend to the boundary. This phenomenon is called height jump, and is first discovered by Hain in \cite{Hai13}. It is studied by Brosnan and Pearlstein carefully. We quote their results \cite[Property 18, Theorem 22, Proposition 140]{BP} in the following. 

Let $ H $ be a torsion-free variation of pure Hodge structure of weight $ -1 $ over the punctured disk $ \Delta^{*r} $, with unipotent monodromy. Let $ V $ be an admissible biextension variation such that $ Gr_{-1}V = H $, then there is a unique degree $ 1 $ homogeneous rational function $ \mu(V)\in \Q(x_1,\ldots,x_r) $ with the following property. 

\begin{theorem}[Brosnan--Pearlstein]\label{ThmBP}
	Let $ \phi:\Delta\to \Delta^r $ be a holomorphic map such that $ \phi(0)=0 $ and $ \phi(\Delta^*)\subset \Delta^{*r} $. Suppose that $ \phi(s)=(\phi_1(s),\ldots,\phi_r(s)) $, and $ \phi_i(s) $ has order $ m_i $ $ m_i>0 $. Then 
	\begin{align*}
	h_{V_{\phi(s)}} = -\mu(V)(m_1,\ldots,m_r) \log |s| +O(1)
	\end{align*}
	for $ s $ close to $ 0 $. 

	Further, the rational function $ \mu(V) $ extends to a continuous homogeneous function on $ \R_{\geqslant 0}^r  $. 
\end{theorem}
The function $ \mu(V) $ is closely related to the asymptotic height pairing in \cite[p.1740]{BP}. They only differed by a linear term in $ m_j $.  By abuse of language, we also call it asymptotic height pairing.

Compare this result with Theorem \ref{Theorem4.12}, we see $ \mu(V) $ could viewed as a function on the deal simplex $ \Delta(Z) $ of $ Z=\Delta^r\backslash (\Delta^{*r}) $, thus its continuity implies that locally it is a Green function of some geometric adelic divisor. Taking globalization, we get the following results. 

Let $ S $ be a smooth quasi-projective variety over $ \C $, and $ \pi:X\to S $ smooth projective. 

\begin{theorem}\label{Theorem6.5}
	The pairing in \ref{Theorem Hain} determines a unique pairing
	\begin{align}\label{GeoPairing}
	CH^p_{\hom}(X/S) \times CH^q_{\hom}(X/S) \to \widetilde{Pic}(S/\C),
	\end{align}
	satisfying the following. 	
	\begin{enumerate}
		\item It is functorial under pull-back $ S'\to S $ for $ S' $ smooth quasi-projective. 
		
		\item If $ \dim S =1 $, the two pairings (\ref{ArchPairing}) and  (\ref{GeoPairing}) coincide in the sense that they induce a pairing
		\begin{align*}
		CH^p_{\hom}(X/S) \times CH^q_{\hom}(X/S) \to \widehat{Pic}(S/\C).
		\end{align*} 
	\end{enumerate}
\end{theorem}

\begin{proof}
	Given any two flat cycles $ Z,W $ of complementary codimension with generically disjoint support, there is a hermitian line bundle $ \bar{L}_{Z,W} $ and a rational section $ \mu_{Z,W} $ such that the underlying line bundle $ L_{Z,W} $ is algebraic and $ \mu_{Z,W} $ is rational. We will show that these information will determine a geometric adelic line bundle. 
	
	We fix a smooth compactification $ \bar{S} $, and let $ U\subset S $ such that $ Z|_U $ and $ W|_U $ has disjoint support. Then we get a function $ h(s)=\log \|\mu_{Z,W}\|_{s} $. 
	
	First we define an geometric adelic divisor $ \tilde{E} \in \hDiv(U/\C) $. Let $ V $ be the biextension variation on $ U $. Since we are taking the geometric case, the biextension variation is admissible, possible with quasi-unipotent monodromy. If we fix a smooth compactification $ \bar{U} $, locally near $ x\in \bar{U}\backslash U $, it has unipotent monodromy after base change. But in the new coordinate, the rational function in Theorem \ref{ThmBP} will only differs by a linear transformation of $ m_i $, so $ \mu(V) $ satisfies the same continuity property on $ \mathbb{R}_{\geqslant 0} $. As a result, $ \mu(V) $ will induce a continuous function on $ \Delta_x $. 
	
	Gluing all local function together, $ \mu(V) $ will induce a continuous function on $ \Delta(\bar{U}\backslash U) $. By pull-back, it can be viewed as a continuous equivariant function on $ U^{\an} $, and thus an geometric adelic divisor $ \tilde{E}\in \widetilde{Div(U/\C)} $. 
	
	Finally we check that $ \tilde{E}\in \widetilde{Div}(S/\C) $. We may assume the $ \bar{U} $ admits a morphism $ f:\bar{U} \to \bar{S} $. Since $ \tilde{E} $ is determined by asymptotic behavior of $ h(s)=\log \|\mu_{Z,W}\| $, and $ \mu_{Z,W} $ is a meromorphic section of line bundle $ f^*L_{Z,W} $,  the Green function  $ \mu(V) $ on $ (f^{-1}(S))^{\beth} $ is just the Green function of the divisor $ \mathrm{\div}(f^*\mu_{Z,W}) $. Thus $ \tilde{E} $ is the pull-back of some adelic divisor on $ S $, which, by abuse of notation, is still denoted by $ E $. 
	
	The functorial property is given by pull-back of variation of mixed Hodge structures. If $ \dim S=1 $, then $ \bar{L} $ extends to a unique hermitian line bundle $ \bar{S} $, and its underlying line bundle on $ \bar{S} $ is just our $ \tilde{L} $. 
\end{proof}

We say that the line bundle $ \tilde{L}_{Z,W} $ parametrizes the asymptotic height pairing.

\subsection{Comparison with geometric Beilinson--Bloch height}
In this section, we show that asymptotic height pairing agrees with the Beilinson--Bloch pairing under certain existence conditions. 

Let $ X $ be a smooth projective variety over a field $ K $. Recall that the group of homologically trivial cycle class is defined by 
\begin{align*}
 CH_{\hom,\ell}(X) := \ker CH^p(X)\to H^{2p}(X,\Q_{\ell}) .
\end{align*} 

If $ \mathrm{char}\ K=0 $, for any cycle $ Z $ in $ X $, take a finitely generated subfield $ k\subset K $ such that both $ X $ and $ Z $ are defined over $ k $, and embed $ k $ into $ \C $, then by the comparison of $ \ell $-adic cohomology and Betti cohomology, $ Z $ is $ \ell $-adic homologically trivial if and only if it's homologically trivial under Betti cohomology. In particular, if $ \mathrm{char}\ K=0 $, then $ CH^{i}_{\hom,\ell}(X) $ does not depend on $ \ell $, and we shall drop $ \ell $ from the notation. 

If $ K=\C $, then we can avoid the $ \ell $-adic cohomology and use Betti cohomology instead. 

If $ K $ is a finitely generated field over $ \C $, we have an equivalent definition. For a cycle $ Z $ in $ X $, choose a model $ \pi:\X\to \mathcal{S} $ and a cycle $ \mathcal{Z} $ in $ \X $, with $ \mathcal{S} $ a smooth quasi-projective variety over $ \C $, $ \pi $ a projective smooth morphism, and $ \mathcal{Z} $ flat over $ S $, such that the generic fiber of $ (\X,\mathcal{S},\mathcal{Z}) $ is $ (X,\Spec K,Z) $. By the proper smooth base change theorem of $ \ell $-adic cohomology, $ Z $ is homologically trivial in $ X $ if and only if for one (hence all) $ s\in S(\C) $, $ \mathcal{Z}_s $ is homologically trivial in $ \mathcal{X}_s $, and $ \mathcal{X}_s $ is a complex projective smooth variety, so we can use Betti cohomology instead.

Let $ S $ be a smooth curve over $ \C $, $ \pi:X\to S $ be a smooth projective morphism, $ \bar{X} $, $ \bar{S} $ be smooth projective compactification of $ X,S $, and $ \bar{\pi}: \bar{X}\to \bar{S} $ be extension of $ \pi $. Clearly $ \bar{\pi} $ is flat.

Let $ Z\in CH^p_{\hom}(X/S) $ with complementary codimension $ p,q $. 	Assume that there exists a projective model $ \bar{\pi}: \bar{X}\to \bar{S} $ with $ \bar{X} $, $ \bar{S} $ smooth, and an extension $ \tilde{Z}\in CH^p(\bar{X}) $ such that the image of $ \tilde{Z} $ lies in the kernel of the composition map
\begin{align*}
CH^p(\bar{X})^{0} := \ker \left(CH^p(\bar{X}) \to H^{2p}(\bar{X},\C) \to H^0(\bar{S}, R^{2p}\bar{\pi}_* \C )\right).
\end{align*}

Then the Beilinson--Bloch height of $ Z $ and $ W $ exists as in \cite[1.2]{Bei} and \cite[Proposition 2.6]{Kah}. In \cite[\S 2E]{Kah}, they showed it coincides with the cohomological definition in \cite{RS}.
\begin{assumption}\label{AssumpExtension}
	Any flat cycle $ Z\in CH^p(X/S) $ admits an extension $ \tilde{Z}\in CH^p(\bar{X})^0 $.
\end{assumption}

\begin{theorem}\label{Thm6.7}
	If the Assumption \ref{AssumpExtension} above holds, the asymptotic height pairing agrees with the Beilinson--Bloch height pairing. In other words, choose a local coordinate near a point $ s\in \bar{S}\backslash S $, assume that $ 0 $ is a singular fiber, we have
	\begin{align*}
		\langle Z_s,W_s \rangle_{s,ar} = -\langle Z,W \rangle_{0,na} \log |s| + O(1)
	\end{align*}
	for $ s $ close to $ 0 $. Here $ \langle Z_s,W_s \rangle_{s,ar} $ means the archimedean Beilinson--Bloch height at the smooth fiber $ X_s $, and $ \langle Z,W \rangle_{0,na} $ means the geometric Beilinson--Bloch height for the field $ K=K(S) $ and the place $ 0 $ with bad reduction.
	
	The error term $ O(1) $ is a harmonic function near $ 0 $, and in particular, bounded near $ 0 $. 
\end{theorem}

In the following, we always assume that the Assumption \ref{AssumpExtension} holds. Let $ \bar{Z} $, $ \bar{W} $ be the Zariski closure of $ Z,W $ on $ \bar{X} $.

\begin{lemma}
	There exists another model $ \bar{X}' $ with an extension $ \tilde{Z}'\in CH^p( \bar{X}')^0 $, such that $ \widetilde{Z}' $ intersects $ \bar{W} $ properly.  
\end{lemma}

\begin{proof}
	Let $ \bar{Z}_i $, $ \bar{W}_j $ are irreducible components of $ Z $ and $ W $. By successive blowup of $ \bar{Z}_i\cap \bar{W}_j $ and taking proper transform, we may assume that $ \bar{Z}\cap \bar{W} $ is empty. Use resolution of singularity again, we assume that $ X $ is smooth with $ X_0 $ normal crossing divisor. Then the pull-back of $ \tilde{Z} $ (which is a cycle class) could be represented by a cycle satisfying the Assumption \ref{AssumpExtension}. So we may assume that $ \bar{Z} $ and $ \bar{W} $ has no intersection. Write $ \tilde{Z} = \bar{Z} +Z_0 $, then $ Z_0 $ is a linear combination of $ i_{j,*} Y $ for some cycle $ Y\in CH^{p-1} (X_{0,j}) $, where $ X_{0,j} $ is a irreducible component of the singular fiber $ X_0 $ of dimension $ n $ and $ i:X_{0,j}\to X $ be the injection. 
	
	Since each $ X_{0,j} $ is smooth of codimension $ 1 $ in $ \bar{X} $, it is a divisor on $ \bar{X} $, hence its scheme intersection with $ W $ coincide with cycle intersection $ i_j^* \bar{W} $. Thus for cycle $ Y $ in $ X_{0,i} $, we have projection formula
	\begin{align*}
	\bar{W} \cdot i_{j,*} Y = i_{j,*}(i_j^* \bar{W}\cdot Y)
	\end{align*}
	Here $ \cdot $ means the intersection in Chow group. By moving lemma in \cite[\S 11.4]{Ful98}, there is a another cycle $ Y' $ in $ X_{0,j} $ which intersects $ i_j^* \bar{W} $ properly, and we take $ i_{j,*}Y' $ to replace $ Y $. 
\end{proof}

Now we will compare the archimedean and non-archimedean height. We will use arithmetic intersection in \cite{SABK}, and try to follow their the notations of Green currents. 

\begin{lemma}\label{Lemma6.9}
	For each point $ s_0\in \bar{S}\backslash S $, there is a local analytic coordinate $ U\cong \Delta $ centered at $ s_0 $, and a $ (p-1,p-1) $-current on  $ g_{\tilde{Z}} $ of $ \tilde{Z} $ on $ \bar{\pi}^{-1} (U) $, such that it is a smooth form away from $ |Z| $, and satisfied the current equation
	\begin{align*}
	dd^c g_{\tilde{Z}} = - \delta_{\tilde{Z}}.
	\end{align*}
\end{lemma}

\begin{proof}
	First we choose any Green form $ g'_{\tilde{Z}} $ of $ \tilde{Z} $ as in \cite[\S 2]{SABK}. Then there is a smooth $ (p,p) $-form $ \omega^{p,p} $ on $ X $, such that 
	\begin{align*}
	dd^c g'_{\tilde{Z}} = -\delta_{\tilde{Z}} + \omega^{p,p}.
	\end{align*}

	By our assumption on $ \tilde{Z} $, there exist an analytic neighborhood $ U $  of $ s_0 $ such that both $ dd^c g_{\tilde{Z}}  $ and $ \delta_{\tilde{Z}} $ vanish in the current cohomology $ H^{2p}_{cur}(\pi^{-1}(U),\C) $, which agrees with the Betti or de Rham cohomology $ H^{2p}(\pi^{-1}(U),\C) $. Therefore, there is a smooth $ (p-1,p-1) $-form $ \omega^{p-1,p-1} $ such that $ dd^c\omega^{p-1,p-1} = \omega^{p,p} $. Then just take $ g_{\tilde{Z}} = g'_{\tilde{Z}} - \omega^{p-1,p-1} $. 
\end{proof}

Now let $ \tilde{Z}\in CH^p(\bar{X})^0 $, and $ W\in CH^q_{\hom}(X/S) $. Let $ \bar{W} $ be its Zariski closer. We may assume that $ \tilde{Z} $ and $ \bar{W} $ intersects properly. 

Since now the theorem is pure analytically local, we may assume $ \bar{S}=\Delta $, the unit disk, and $ S=\Delta^* $. Let $ g_{\tilde{Z}} $ be the current in Lemma \ref{Lemma6.9}. Take any other Green form $ g_{\bar{W}} $ of $ \bar{W} $, and take arithmetic intersection $ (\tilde{Z},g_{\tilde{Z}}) * (\bar{W},g_{\bar{W}}) $ of \cite[II.3]{SABK}, which is of the form $ (\tilde{Z}\cdot W, g_{\tilde{Z}} * g_{\bar{W}}) $ satisfying the equation
\begin{align*}
dd^c (g_{\tilde{Z}} * g_{\bar{W}}) = -\delta_{\tilde{Z}\cdot \bar{W}} 
\end{align*}
by \cite[II.3, Theorem 4]{SABK}, where $ \tilde{Z}\cdot \bar{W} $ is the intersection product with Serre's mulplicity formula. 

Taking push forward, we get an equation of current
\begin{align*}
dd^c (\pi_* g_{\tilde{Z}} * g_{\bar{W}}) = -\delta_{\pi_*(\tilde{Z}\cdot \bar{W})} = - \langle Z,W \rangle_{s_0,\mathrm{na}} [s_0]. 
\end{align*}

\begin{lemma}
	The $ (0,0) $-current $ \pi_* (g_{\tilde{Z}} * g_{\bar{W}}) $ is represented by the function $ h(s) = 2\langle Z_s, W_s \rangle_{a,\mathrm{ar}} $. 
\end{lemma}

\begin{proof}
	By direct computation, for any compact support smooth  $ (1,1) $-form $ f(s) \d s\wedge \d \bar{s} $, we have
	\begin{align*}
	&\qquad \ \langle \pi_* g_{\tilde{Z}} * g_{\bar{W}}, f(s) \d s\wedge \d \bar{s} \rangle \\
	& = \int_{\bar{W}} g_{\tilde{Z}} \wedge \pi^*(f(s) \d s\wedge \d \bar{s})\\
	& = \int_{\Delta^*} \left(\int_{W_s} g_{\tilde{Z}} \right) f(s) \d s \wedge \d \bar{s} \\
	&= \int_{\Delta} 2\langle Z_s,W_s \rangle_{s,\mathrm{ar}} f(s) \d s \wedge \d \bar{s} .
	\end{align*}
\end{proof}

Now we get the current equation 
\begin{align*}
 \d \d^c\  2\langle Z_s,W_s \rangle_{s,\mathrm{ar}} = - \langle Z,W \rangle_{s_0,\mathrm{na}} [s_0] .
\end{align*}
Solve this equation, we get
\begin{align*}
\langle Z_s,W_s \rangle_{s,\mathrm{ar}} = -\langle Z,W \rangle_{s_0,\mathrm{na}}\log |s| +O(1),
\end{align*}
here $ O(1) $ is a harmonic near $ s_0 $, and in particular, bounded. Thus we get Theorem \ref{Thm6.7}. 

\begin{remark}
	In \cite{SABK}, they only define push forward for projective smooth morphism. Here we actually use push forward for flat morphism. 
	
	Let $ f:X\to S $ is projective flat and generically smooth, and $ (Z,g_Z) $ be an arithmetic Chow cycle, i.e., $ \d \d^c g_Z = -\delta_Z + \omega $ for some smooth form $ \omega $. If we apply push forward, we will get an equation
	\begin{align*}
	dd^c f_* g_Z = -f_*\delta_{Z} + f_*\omega. 
	\end{align*}
	However, for non-smooth $ f $, the push forward $ f_*\omega $ will not be a smooth form on $ S $ anymore. Instead, it is a smooth form on the smooth locus of $ f $, and along the boundary, it may have singularity. 
\end{remark}

\end{document}